\documentclass[11pt,a4paper]{article}

\usepackage{graphicx}
\usepackage{array}
\usepackage{latexsym}
\usepackage{amsmath,amsthm}
\usepackage{amssymb}
\usepackage{amsfonts,cite}
\usepackage{algorithm}
\usepackage{algpseudocode}
\usepackage{booktabs}
\usepackage[cm]{fullpage}
\usepackage{wrapfig}
\usepackage{color,cancel}
\usepackage{enumitem}

\usepackage{caption}
\usepackage{subcaption}
\usepackage{mathrsfs}
\usepackage{hyperref}
\usepackage{xcolor}

\newtheorem{theorem}{Theorem}[section]

\newtheorem{definition}{Definition}[section]
\newtheorem{example}{Example}[section]

\newtheorem{lemma}{Lemma}[section]

\newtheorem{remark}{Remark}[section]

\DeclareMathOperator{\argmin}{argmin}

\DeclareMathOperator{\dist}{dist}

\newcommand{\norm}[1]{\left\lVert#1\right\rVert}
\newcommand{\dprod}[1]{\left\langle#1\right\rangle}

\newcommand{\RR}
{\ensuremath{\mathbb{R}}}
\newcommand{\PP}
{\ensuremath{\mathbb{P}}}

\newcommand{\phib}
{\ensuremath{\boldsymbol{\phi}}}

\newcommand{\yb}
{\ensuremath{\boldsymbol{y}}}
\newcommand{\NN}
{\ensuremath{\mathbb{N}}}
\newcommand{\Fscr}{\ensuremath{\mathscr{F}}}
\newcommand{\Pas}{\ensuremath{(\mathbb{P}\text{-a.s.})}}
\newcommand{\Hc}{\ensuremath{\mathcal{H}}}
\newcommand{\Fc}{\ensuremath{\mathcal{F}}}
\newcommand{\EE}{\ensuremath{\mathbb{E}}}

\title{Inexact and Stochastic Gradient Optimization Algorithms with Inertia and Hessian Driven Damping}

\author{
Harsh Choudhary\thanks{Department of Computer Science and Engineering, Czech Technical University, Prague.
\url{choudharyharsh122@gmail.com}}  \and 
Jalal Fadili\thanks{Normandie Univ, ENSICAEN, CNRS, GREYC, Caen, France.
\url{Jalal.Fadili@greyc.ensicaen.fr}} \and  
Vyacheslav Kungurtsev\thanks{, Prague. \url{vyacheslav.kungurtsev@fel.cvut.cz}} 
}

\begin{document}

\maketitle

\begin{abstract}
In a real Hilbert space setting, we study the convergence properties of an inexact gradient algorithm featuring both viscous and Hessian driven damping for convex differentiable optimization. In this algorithm, the gradient evaluation can be subject to deterministic and stochastic perturbations. In the deterministic case, we show that under appropriate summability assumptions on the perturbation, our algorithm enjoys fast convergence of the objective values, of the gradients and weak convergence of the iterates toward a minimizer of the objective. In the stochastic case, assuming the perturbation is zero-mean, we can weaken our summability assumptions on the error variance and provide fast convergence of the values both in expectation and almost surely. We also improve the convergence rates from $\mathcal{O}(\cdot)$ to $o(\cdot)$ in almost sure sense. We also prove almost sure summability property of the gradients, which implies the almost sure fast convergence of the gradients towards zero. We will highlight the trade-off between fast convergence and the applicable regime on the sequence of errors in the gradient computations. We finally report some numerical results to support our findings.
\end{abstract}

\section{Introduction}\label{s:intro}
\subsection{Problem statement and motivations}
In this paper, we seek to solve
\begin{equation}\label{eq:mainprob}
\min_{x \in \Hc} f(x) ,
\end{equation}
where $\Hc$ is a real Hilbert space with inner product $\dprod{\cdot,\cdot}$ and associated norm $\norm{\cdot}$. Throughout the paper, we make the following standing assumptions:
\begin{equation}\tag{$\mathrm{H}_f$}\label{eq:mainassump}
\begin{cases}
\text{$f$ is a convex smooth function on $\Hc$}, \\
\text{$\nabla f$, the gradient of $f$, is $L$-Lipschitz continuous}, \\ 
\text{and $S := \argmin_{\Hc} f\neq \emptyset$.}
\end{cases}
\end{equation}

Our aim is to propose and study an inertial gradient algorithm, featuring both viscous damping (friction) and Hessian-driven geometric damping, where the first-order gradient information is only accessible up to some additive error, either deterministic or stochastic.

One of our motivations for this additive perturbation model comes from optimization, where the gradient may be accessed only inaccurately, either because of physical, numerical or computational reasons. The prototype example we think of is the stochastic optimization problem where
\begin{equation}\label{eq:minExpecobj}
f(x) = \int_{\Xi} F(x,\xi)d\mu(\xi) =: \EE_{\xi \sim \mu}[F(x,\xi)] ,\,F:\Hc\times \Xi\to \RR
\end{equation}
wherein $\xi \sim \mu$ is a random variable taking values in $\Xi$, $F(x,\cdot)$ is $\mu$-integrable for each $x \in \Hc$, and $F(\cdot,\xi) \in C^1(\Hc)$ and convex for each $\xi$. Stochastic optimization is a prominent
and active field of research for solving problems arising in many applications including machine learning and signal processing. Computing $\nabla f(x)$ (or even $f(x)$) is either impossible or computationally very expensive. Rather, one draws $m$ independent samples of $\xi$, say $(\xi_i)_{1 \leq i \leq m}$, and compute the average estimate
\begin{equation}\label{eq:deff}
\widehat{\nabla f}(x) = \frac{1}{N} \sum_{i=1}^N \nabla F(x,\xi_i) .
\end{equation}
In our setting, the error at iteration $k$ of an algorithm based on the first-order information $\nabla f(x_k)$ is then $e_k = \nabla f(x_k) - \widehat{\nabla f}(x_k)$. Observe that conditioned on $x_k$, and by independent sampling, $e_k$ has indeed zero-mean and its variance is the square of the standard error associated with a sample average, which scales as $\mathcal{O}(1/N)$. 
Thus, to make this variance verify appropriate summability assumptions in $k$, that will be made clear in our analysis, one has to take $N$ increase fast enough with $k$ (see our theoretical results as well as the numerical section for a detailed example). Our analysis will also reveal that under the weaker condition that $\|e_k\|$ vanishes (in expectation), we will establish convergence (and rate of convergence) of the objective values to a ``noise-dominated region" around the optimal value, and we will characterize precisely the size of this region. Of course, summability is stronger. However, one has to keep in mind that our goal is to establish additional convergence guarantees, including fast convergence of the gradient and $o(\cdot)$ rates. Deriving such guarantees is more challenging and cannot be proved under a mere smallness condition on the error.

\subsection{Previous work}
\paragraph{Exact inertial dynamics and algorithms}
Damped inertial dynamics have a natural mechanical and physical interpretation. Asymptotically, they tend to stabilize the system at a minimizer of the global energy function. As such, they offer an intuitive way to develop fast optimization methods. In \cite{polyak1964some}, B. Polyak initiated the use of inertial dynamics to accelerate the gradient method in optimization based on a second order in time inertial dynamical system with a fixed viscous damping coefficient; the so-called Heavy Ball with Friction (HBF) method. (HBF) provided momentum to the dynamics that accelerated the convergence for strongly convex objectives. In 1983, Nesterov introduced a momentum method, known in the literature as Nesterov accelerated gradient (NAG for short) \cite{nesterov1983method}. The Ravine method (RAG for short) was introduced by Gelfand and Tsetlin in 1961 \cite{GelfandRavine}. For a long time, the RAG and NAG algorithms were confused with each other and only recently, connections between the two methods were brought to the forefront in \cite{AF-Ravine-SIOPT}.

To obtain a continuous ODE surrogate of the NAG algorithm, a crucial step was taken by Su-Boyd-Cand\`es in 2016 \cite{su2014differential}. They introduced an asymptotic vanishing damping coefficient of the form $ \frac{\alpha}{t}$, where $\alpha > 0$ and $t > 0$ represents the time variable in the inertial system. In particular, for a general $L$-smooth convex function $f$, the condition $\alpha > 3$ guarantees the asymptotic convergence rate of the values with a rate of $o \left( 1/t^2 \right)$, as well as the weak convergence of each trajectory toward an optimal solution \cite{ACPR}. Results in the algorithmic case can be found in \cite{adly2024complexity}. In recent years, there has been an in-depth study linking the NAG method to inertial dynamics with vanishing viscous damping; see \cite{AF-Ravine-SIOPT,adly2024accelerated,AAF,polyak2020accelerated,attouch2024fast,CEG2,AAD,ACR-subcrit}

The above inertial systems may suffer from transverse oscillations, and it is desirable to dampen them. This is precisely the motivation behind the introduction of geometric Hessian-driven damping in \cite{ACFR}.
Several recent studies have been devoted to inertial dynamics that combine asymptotic vanishing damping with geometric Hessian-driven damping (sometimes called Newton-type inertial dynamics). See, for example, \cite{APR,APR1,ACFR,SDJS}. For instance, the inertial system 
\begin{equation}\label{eq:isehd}
\ddot{x}(t) + \displaystyle{\frac{\alpha}{t}}\dot{x}(t) +  \beta(t) \frac{d}{dt}\nabla  f (x(t)) + \nabla  f (x(t)) = 0,
\end{equation}
was introduced in \cite{APR} and studied in detail in~\cite{ACFR}. It was found to have favorable fast convergence properties in both the objective values as well as the gradients.

In \cite{ACFR,ACFRiterates}, the IGAHD algorithm, as a gradient-type discretization of \eqref{eq:isehd} was introduced and its convergence properties were analyzed, including the reduction of the oscillatory behaviour. A detailed comparative study of the IGAHD, NAG, and RAG methods was reported in \cite{AAF}.



\paragraph{Inertial dynamics and algorithms with deterministic and stochastic perturbations}
Due to the importance of the subject in optimization and control, several papers have been devoted to the study of perturbations in dissipative inertial systems and in the corresponding accelerated first-order algorithms. This was first considered in the case of a fixed viscous damping, \cite{acz,HJ1}. Then it was studied within the framework of the NAG, and of the corresponding inertial dynamics with vanishing viscous damping, see \cite{SDJS,AC2R-JOTA,ACPR,AD15,SLB,VSBV}. In the presence of an additional Hessian driven damping, first results have been obtained in \cite{APR1,ACFR,ACFRiterates} in the case of a smooth function. In~\cite{attouch2021effect}, the dynamical system \eqref{eq:isehd} with (integrable) errors was considered, and it was found that the favorable convergence properties of this method was robust, i.e., the rates were maintained. This form of
analysis, of considering the asymptotic dynamics of an ODE using Lyapunov theory, is insightful for the understanding and development
of optimization algorithms. These results were extended to stochastic continuous time inertial dynamics in \cite{maulen2024stochastic, maulen2024inertial}

Stochastic gradient descent methods with inertia/momentum are widely used in optimization subroutines in many applications. These algorithms are the subject of an active research work to understand their convergence behavior and were studied in several works, focusing exclusively on stochastic versions of the HBF and NAG/RAG methods; see
\cite{lan2012optimal,Lin2015,frostig2015regularizing,ghadimi2016mini, Gupta24,kidambi2018insufficiency,AR,AZ,Yang2016,gadat2018stochastic, loizou2020momentum, Laborde2020,LanBook,Sebbouh2021,Driggs22,AFK24}.


    \subsection{Contributions and relation to prior work}
    In this work, we analyze a perturbed version, in the deterministic and stochastic settings, of the IGAHD algorithm proposed in \cite{ACFR} 
    for solving infinite-dimensional optimization problems of the form \eqref{eq:mainprob} over a real Hilbert space. We are not aware of any work handling perturbations for inertial algorithms featuring both viscous and Hessian driven damping terms. Our work appears to be the first in this direction. This paves the way of using the stochastic version of such an algorithm for optimization problems of the form \eqref{eq:minExpecobj}. Our contributions can summarized as follows:
    \begin{itemize}
    \item Deterministic errors: we show that under appropriate summability assumptions on the perturbation, our algorithm enjoys fast convergence of the objective values, of the gradients and weak convergence of the iterates toward a minimizer of the objective. This extends to the algorithmic case our previous work in \cite{attouch2021effect} on the deterministic perturbed version of \eqref{eq:isehd}.
    \item Stochastic errors: assuming the perturbation is a random variable that is zero-mean, we can assume relatively weak conditions on the error variance and provide fast convergence of the values both in expectation and almost surely. We also improve the convergence rates from $\mathcal{O}(\cdot)$ to $o(\cdot)$ in expectation and almost sure sense. We also prove almost sure summability property of the gradients, which implies the almost sure fast convergence of the gradients towards zero. 
    \item Speed-accuracy tradeoff: throughout, we will highlight and discuss the tradeoff between fast convergence, the choice of the algorithm parameters, and the required properties on the sequence of errors in the gradient computations.
    \end{itemize}
    
    \subsection{Paper organization}
    The rest of the paper is organized as follows. In Section~\ref{s:conv} we
    first introduce the inexact IGAHD algorithm. We then focus on deterministic errors/noise and carry out a Lyapunov analysis before deriving its convergence results (rates and weak convergence) under appropriate summability assumptions on the errors. We also characterize the corresponding noise-dominated region. Section~\ref{s:stoch} will be devoted to the case of stochastic errors/noise. The latter is not a mere extension of the results in Section~\ref{s:conv} and requires new and different arguments. Most prominently, the estimates of the gradients will not be independent and this has to be addressed very carefully.
    Section~\ref{s:num} reports some numerical results to support our theoretical findings before finally drawing some general insights and conclusions in Section~\ref{s:conc}.




\section{Deterministic Errors}\label{s:conv}
We begin by considering the case of deterministic errors and derive convergence rates in the case of their integrability. In particular, algorithmically we consider that every attempted evaluation of $\nabla f(x)$ is subject to an error. The main steps are summarized in Algorithm~\ref{alg:hdde}.


\begin{algorithm}[H]
\caption{Inexact Inertial Gradient Algorithm with Hessian Driven Damping (I-IGAHD)}\label{alg:hdde}
\textbf{Input:} Initial point $x_1=x_{0}\in\Hc$, step-sizes $s_k \in ]0,1/L]$, parameters $(\alpha,\beta_k)$ satisfying $\alpha \geq 3$, $0 \leq \beta_k <2\sqrt{s_k}$.
\begin{algorithmic}[1]
\For{$k=1,2,...$}
\State Set $\alpha_k=1-\frac{\alpha}{k}$. 
\State Set,
\begin{equation}\label{eq:genalgstepy}
y_k = x_k+\alpha_k(x_k-x_{k-1})-\beta_k \sqrt{s_k}\left(\nabla f(x_k)+M^x_k\right)+\beta_{k-1} \sqrt{s_{k-1}}\left(1-\frac{1}{k}\right)\left(\nabla f(x_{k-1})+M^x_{k-1}\right) .
\end{equation}
\State Set,
\begin{equation}\label{eq:genalgstepx}
x_{k+1}=y_k- s_k (\nabla f(y_k)+M^y_k) .
\end{equation}
\EndFor
\end{algorithmic}
\end{algorithm}

\subsection{Main estimate}
We now proceed to derive a Lyapunov analysis for Algorithm~\ref{alg:hdde}. To this end, let $x^\star \in S$ and define the energy function
\[
E_{k} := s_k t_k^2 (f(x_k)-\min_{\Hc} f)+\frac{1}{2}\|v_k\|^2-\sum\limits_{i=k}^\infty \langle v_{i+1}, M_{i}\rangle , \qquad \forall k \geq 1 ,
\]
where
\begin{align*}
t_{k+1} &:=\frac{k}{\alpha-1}\\
v_k &:= (x_{k-1}-x^\star)+t_k(x_k-x_{k-1}+\beta_{k-1}\sqrt{s_{k-1}}\nabla f(x_{k-1})) \\
M_k &:= \left(\beta_k\sqrt{s_{k}}t_{k+1}M^x_k - \beta_{k-1}\sqrt{s_{k-1}}t_k M^x_{k-1}\right) + s_kt_{k+1}M^y_k .
\end{align*}
We will take $\alpha_0=t_0=0$. We have the useful properties that we will used repeatedly:
\begin{equation} \label{eq:k1}
\begin{aligned}
    t_{k+1}^2 - t_{k}^2 \leq t_{k+1}, \qquad \text{and} \qquad t_{k+1}\alpha_k = t_k - 1 .
\end{aligned}
\end{equation}

We will also use the shorthand notation
\begin{equation}\label{eq:f_k_def}
    \begin{aligned}
       f_k(x) := f(x)+\langle M^y_k,x\rangle .        
    \end{aligned}
\end{equation}

\begin{lemma}\label{lem:lyaperr}
Let $f$ satisfy \eqref{eq:mainassump} and $(x_k)_{k \in \NN}$ be a sequence generated by Algorithm~\ref{alg:hdde}, with $\alpha\ge 3$, $0\leq\beta_k\leq\frac{\sqrt{s_{k}}}{2}$ and $s_{k} \in ]0,1/L]$ is nonincreasing.
Then $\exists k_0 \geq 1$ large enough and a sequence $(\epsilon_k)_{k \in \NN}$ obeying $\epsilon_k \in [0,\beta_k(2\sqrt{s_{k}} - \beta_k)]$ such that for all $k \geq k_0$, we have
\begin{multline}\label{eq:genrecursion}
E_{k+1} - E_{k} 
\leq s_k\left(t_{k+1}^2 - t_{k+1} - t_{k}^2 \right)\left( f(x_k) - \min_{\Hc} f\right)
- \frac{\epsilon_k}{4}s_kt_{k+1}^2 \left\|\nabla f(y_k)\right\|^2 \\
+ \frac{s_k^2}{2}t_{k+1}^2\left\|M^y_k\right\|^2 
+ \beta_k^2s_kt_{k+1}^2\left(\frac{2s_k}{\epsilon_k} + 1 \right)\left\|M^x_k\right\|^2 
+ \beta_{k-1}^2s_{k-1}t_{k}^2\left(\frac{2s_k}{\epsilon_k} + 1 \right)\left\|M^x_{k-1}\right\|^2 .
\end{multline}
\end{lemma}

\noindent Before we prove the result, a few important observations are in order, that will be useful for the derivations in the sequel.
\begin{remark}\label{rem:lemma}
Observe that since $\alpha \geq 3$, we have $t_{k+1}^2 - t_{k+1} \leq t_{k}^2$, and thus the first term in the rhs of \eqref{eq:genrecursion} is negative. Moreover, if $\inf_{k} s_k > 0$, then one can take $0 < \beta_k \leq \sqrt{s_k}/2$, and, in turn, $\inf_k \epsilon_k > 0$. This remark is important in view of proving fast summability of the gradients.
\end{remark}




\begin{proof}
Applying Lemma~\ref{lemma:A1} to the function $f_k$ at $y=y_k$ and $x=x_k$, and then at $y=y_k$ and $x=x^\star$, and using $x_{k+1}=y_k-s_{k}\nabla f_k(y_k)$ and $\nabla f_k(x^\star)= M^y_k$, we get
\begin{align}
f_k(x_{k+1}) &\le f_k(x_k)+\langle \nabla f_k(y_k),y_k-x_k\rangle -\frac{s_{k}}{2}\|\nabla f_k(y_k)\|^2 -\frac{s_{k}}{2}\|\nabla f_k(x_k)-\nabla f_k(y_k)\|^2 \nonumber\\
&= f_k(x_k)+\langle \nabla f_k(y_k),y_k-x_k\rangle -\frac{s_{k}}{2}\|\nabla f_k(y_k)\|^2 -\frac{s_{k}}{2}\|\nabla f(x_k)-\nabla f(y_k)\|^2\label{eq:desclemmaa} \\
f_k(x_{k+1}) &\le f_k(x^\star)+\langle\nabla f_k(y_k),y_k-x^\star\rangle-\frac{s_{k}}{2}\|\nabla f_k(y_k)\|^2 -\frac{s_{k}}{2}\|\nabla f_k(y_k)-M^y_k\|^2 \nonumber\\
&\le f_k(x^\star)+\langle\nabla f_k(y_k),y_k-x^\star\rangle-\frac{s_{k}}{2}\|\nabla f_k(y_k)\|^2 -\frac{s_{k}}{2}\|\nabla f(y_k)\|^2.\label{eq:desclemmab}
\end{align}
where we used that $\nabla f_k(x) = \nabla f(x) + M^y_k$ for any $x$. Multiplying~\eqref{eq:desclemmaa} by $t_{k+1}-1\ge 0$ and adding~\eqref{eq:desclemmab}, we obtain
\begin{align*}
t_{k+1}(f_k(x_{k+1})-f_k(x^\star)) 
\le & (t_{k+1}-1)(f_k(x_k)-f_k(x^\star))\\ & +\langle \nabla f_k(y_k),(t_{k+1}-1)(y_k-x_k)+y_k-x^\star\rangle  \\ 
& -\frac{s_{k}}{2}t_{k+1} \|\nabla f_k(y_k)\|^2  - \frac{s_{k}}{2}(t_{k+1}-1) \|\nabla f(x_k)-\nabla f(y_k)\|^2 \\ 
& -\frac{s_{k}}{2}\left\|\nabla f(y_k)\right\|^2 \\
= & (t_{k+1}-1)(f_k(x_k)-f_k(x^\star))\\ & +\langle \nabla f_k(y_k),(t_{k+1}-1)(y_k-x_k)+y_k-x^\star\rangle  \\ 
& -\frac{s_{k}}{2}(t_{k+1}+1) \|\nabla f(y_k)\|^2  -\frac{s_{k}}{2}(t_{k+1}-1) \|\nabla f(x_k)-\nabla f(y_k)\|^2 \\ 
& -s_{k}t_{k+1}\langle \nabla f(y_k),M^y_k\rangle-\frac{s_{k}}{2}t_{k+1}\|M^y_k\|^2 .
\end{align*}
We then multiply both sides by $t_{k+1}$ to arrive at
\begin{equation}\label{eq:descUtil}
\begin{aligned}
    t_{k+1}^2 \left( f_k(x_{k+1}) - f_k(x^\star) \right) 
    &\leq  t_k^2 \left( f_k(x_k) - f_k(x^\star) \right)
    + (t_{k+1}^2 - t_{k+1} - t_k^2) \left( f_k(x_k) - f_k(x^\star) \right) \\
    &+ t_{k+1} \langle \nabla f_k(y_k), (t_{k+1} - 1)(y_k - x_k) + y_k - x^\star \rangle \\
    & - \frac{s_{k}}{2} t_{k+1} (t_{k+1} + 1) \| \nabla f(y_k) \|^2 \\
    & - \frac{s_{k}}{2} t_{k+1}(t_{k+1} - 1) \| \nabla f(x_k) - \nabla f(y_k) \|^2 \\
    & - s_{k} t_{k+1}^2 \langle \nabla f(y_k), M^y_k \rangle - \frac{s_{k} }{2}t_{k+1}^2 \| M^y_k \|^2 .
\end{aligned}
\end{equation}
Recalling the definition of $f_k$, plugging it in the last inequality and rearranging the terms (where we shifted all the error terms to the rightmost part of the expression) we get 
\begin{equation}\label{eq:stoch_util}
\begin{aligned}
t_{k+1}^{2}\left(f\left(x_{k+1}\right) - f\left(x^\star\right)\right) 
&\leq t_{k}^{2}\left(f\left(x_{k}\right) - f\left(x^\star\right)\right)
 + \left(t_{k+1}^{2} - t_{k+1} - t_{k}^{2}\right)\left(f\left(x_{k}\right) - f\left(x^\star\right)\right) \\
& + t_{k+1}\left\langle \nabla f(y_{k}), \left(t_{k+1} - 1\right)\left(y_{k} - x_{k}\right) + y_{k} - x^\star \right\rangle\\
& - \frac{s_{k}}{2} t_{k+1} \left(t_{k+1} + 1\right) \left\|\nabla f\left(y_{k}\right)\right\|^{2} \\
& - \frac{s_{k}}{2} t_{k+1} \left(t_{k+1} - 1\right) \left\|\nabla f\left(x_{k}\right) - \nabla f\left(y_{k}\right)\right\|^{2} \\
& - \left\langle M^y_k , t_{k+1}^{2}\left(x_{k+1} - x^\star\right) - \left(t_{k+1}^{2} - t_{k+1}\right)\left(x_{k} - x^\star\right) \right\rangle \\
& + t_{k+1}\left\langle M^y_k, \left(t_{k+1} - 1\right)\left(y_{k} - x_{k}\right) + y_{k} - x^\star \right\rangle\\
& - s_{k} t_{k+1}^2 \left\langle M^y_k , \nabla f\left(y_{k}\right)\right\rangle - \frac{s_{k}}{2} t_{k+1}^2\left\|M^y_k\right\|^{2} .
\end{aligned}
\end{equation}
Multiplying by $s_k$ on both sides of \eqref{eq:stoch_util}, using that $s_k$ is nonincreasing, we can plug this expression into $E_{k}$ to yield:
\begin{equation}\label{eq:v_k_inter}
\begin{aligned}
E_{k+1} - E_{k} 
\leq& s_k\left(t_{k+1}^2 - t_{k+1} - t_{k}^2 \right)\left( f(x_k) - \min_{\Hc} f\right) \\
& + s_kt_{k+1} \left\langle \nabla f\left(y_{k}\right), \left(t_{k+1} - 1\right)\left(y_{k} - x_{k}\right) + y_{k} - x^\star \right\rangle \\
& - \frac{s_{k}^2}{2} t_{k+1} \left(t_{k+1} + 1\right) \left\|\nabla f\left(y_{k}\right)\right\|^{2} - \frac{s_{k}^2}{2} t_{k+1} \left(t_{k+1} - 1\right) \left\|\nabla f\left(x_{k}\right) - \nabla f\left(y_{k}\right)\right\|^{2} \\
& + \frac{1}{2} \left\|v_{k+1}\right\|^{2} - \frac{1}{2} \left\|v_{k}\right\|^{2} \\ 
&- s_kt_{k+1}\left\langle M^y_k , t_{k+1}\left(x_{k+1} - x_{k}\right) + \left(x_{k} - x^\star\right)\right\rangle \\
& + s_kt_{k+1}\left\langle M^y_k, \left(t_{k+1} - 1\right)\left(y_{k} - x_{k}\right) + y_{k} - x^\star \right\rangle\\
& - s_{k}^2 t_{k+1}^2 \left\langle M^y_k , \nabla f\left(y_{k}\right) \right\rangle - \frac{s_{k}^2}{2}t_{k+1}^2 \left\|M^y_k\right\|^{2} + \left\langle M_{k} , v_{k+1} \right\rangle .
\end{aligned}
\end{equation}
By the well-known three point identity, we have
\[
\frac{1}{2}\left\|v_{k+1}\right\|^{2}-\frac{1}{2}\left\|v_{k}\right\|^{2}=\left\langle v_{k+1}-v_{k}, v_{k+1}\right\rangle-\frac{1}{2}\left\|v_{k+1}-v_{k}\right\|^{2}
\]
By definition of $v_k$, and given the recursion on $x_{k+1}$ and afterwards the expression for $y_k$ Algorithm~\ref{alg:hdde} and that $t_k-1= t_{k+1}\alpha_k$,  we have
\begin{align}\label{eq:vk+1vk}
v_{k+1} - v_{k}
&= x_{k} - x_{k-1} + t_{k+1}(x_{k+1} - x_{k} + \beta_k \sqrt{s_k} \nabla f( x_{k})) - t_k (x_{k} - x_{k-1} + \beta_{k-1} \sqrt{s_{k-1}}  \nabla f( x_{k-1})) \nonumber\\
&= t_{k+1}(x_{k+1} - x_{k}) -(t_k-1)  (x_{k} - x_{k-1}) 
+ \Big( \beta_k\sqrt{s_k} t_{k+1} \nabla f( x_{k})-  \beta_{k-1}\sqrt{s_{k-1}}t_{k}\nabla f( x_{k-1}) \Big)  \nonumber\\
&= t_{k+1}\Big(x_{k+1} - (x_{k}+\alpha_k  (x_{k} - x_{k-1})\Big) + \Big( \beta_k\sqrt{s_k} t_{k+1} \nabla f( x_{k})-  \beta_{k-1}\sqrt{s_{k-1}}t_{k}\nabla f( x_{k-1}) \Big) \nonumber\\
&= t_{k+1}\left(x_{k+1} - y_k\right) -\beta_k \sqrt{s_k}t_{k+1}\left(\nabla f(x_k)+M^x_k\right)+\beta_{k-1} \sqrt{s_{k-1}}t_{k+1}\left(1-\frac{1}{k}\right)\left(\nabla f(x_{k-1})+M^x_{k-1}\right) \nonumber\\
&+ \Big( \beta_{k}\sqrt{s_k} t_{k+1} \nabla f( x_{k})-  \beta_{k-1}\sqrt{s_{k-1}}t_{k}\nabla f( x_{k-1}) \Big) \nonumber\\
&= -s_kt_{k+1}\left(\nabla f(y_k) + M^y_k\right) + \beta_{k-1} \sqrt{s_{k-1}}\left(t_{k+1}\left(1 - \frac{1}{k} \right)-t_k\right)\nabla f(x_{k-1}) \nonumber\\
&-\beta_{k}\sqrt{s_{k}}t_{k+1}M^x_k + \beta_{k-1}\sqrt{s_{k-1}}t_{k+1}\left(1-\frac{1}{k}\right)M^x_{k-1} \nonumber\\
&= -s_kt_{k+1}\nabla f(y_k) - \left(\beta_{k}\sqrt{s_{k}}t_{k+1}M^x_k - \beta_{k-1}\sqrt{s_{k-1}}t_k M^x_{k-1}\right) - s_kt_{k+1}M^y_k  \nonumber\\
&= -s_kt_{k+1}\nabla f(y_k) - M_k .
\end{align}
This expression can now be substituted into the previous equation to obtain
\begin{align*}
\frac{1}{2}\left\|v_{k+1}\right\|^{2} - \frac{1}{2}\left\|v_{k}\right\|^{2}
=& -\frac{1}{2}\left\|s_kt_{k+1}\nabla f(y_k)+M_k\right\|^{2} \\
&- \left\langle s_kt_{k+1}\nabla f(y_k) + M_k, (x_{k}-x^\star)+t_{k+1}(x_{k+1}-x_{k}+\beta_{k}\sqrt{s_{k}}\nabla f(x_{k})) \right\rangle \\
=& -\frac{s_k^2t_{k+1}^2}{2}\left\|\nabla f(y_k)\right\|^{2} - s_kt_{k+1}\langle\nabla f(y_k),M_k\rangle - \frac{1}{2}\left\|M_k\right\|^2 \\
& - \left\langle s_kt_{k+1}\nabla f(y_k) + M_k, (x_{k}-x^\star)+t_{k+1}(x_{k+1}-x_{k}+\beta_{k}\sqrt{s_{k}}\nabla f(x_{k})) \right\rangle .
\end{align*}
Plugging this into \eqref{eq:v_k_inter}, we obtain
\begin{equation}\label{eq:a_k_inter}
\begin{aligned}
E_{k+1} - E_{k} 
\leq& s_k\left(t_{k+1}^2 - t_{k+1} - t_{k}^2 \right)\left( f(x_k) - \min_{\Hc} f\right) \\
& + s_kt_{k+1} \left\langle \nabla f\left(y_{k}\right), \left(t_{k+1} - 1\right)\left(y_{k} - x_{k}\right) + y_{k} - x^\star \right\rangle \\
& - \frac{s_{k}^2}{2} t_{k+1} \left(2t_{k+1} + 1\right) \left\|\nabla f\left(y_{k}\right)\right\|^{2} - \frac{s_{k}^2}{2} t_{k+1} \left(t_{k+1} - 1\right) \left\|\nabla f\left(x_{k}\right) - \nabla f\left(y_{k}\right)\right\|^{2} \\
& - s_kt_{k+1} \left\langle \nabla f(y_k) , (x_{k}-x^\star)+t_{k+1}(x_{k+1}-x_{k}+\beta_{k}\sqrt{s_{k}}\nabla f(x_{k})) \right\rangle \\ 
&- s_kt_{k+1}\langle M_k , \nabla f(y_k)\rangle - \frac{1}{2}\left\|M_k\right\|^2 - \left\langle M_k , v_{k+1} \right\rangle \\
&- s_kt_{k+1}\left\langle M^y_k , t_{k+1}\left(x_{k+1} - x_{k}\right) + \left(x_{k} - x^\star\right)\right\rangle \\
& + s_kt_{k+1}\left\langle M^y_k, \left(t_{k+1} - 1\right)\left(y_{k} - x_{k}\right) + y_{k} - x^\star \right\rangle\\
& - s_{k}^2 t_{k+1}^2 \left\langle M^y_k , \nabla f\left(y_{k}\right) \right\rangle - \frac{s_{k}^2}{2}t_{k+1}^2 \left\|M^y_k\right\|^{2} + \left\langle M_{k} , v_{k+1} \right\rangle .
\end{aligned}
\end{equation}
Observe that
\begin{align*}
&(t_{k+1} - 1)(y_k - x_k) + y_k - x_k - t_{k+1}\left(x_{k+1} - x_k + \beta_{k} \sqrt{s_{k}} \nabla f(x_k)\right) \\
&= t_{k+1} (y_k - x_k) - t_{k+1}\left(x_{k+1} - x_k\right) - t_{k+1} \beta_{k} \sqrt{s_{k}} \nabla f(x_k) \\
&= t_{k+1}(y_k - x_{k+1}) - t_{k+1} \beta_{k} \sqrt{s_{k}} \nabla f(x_k) \\
&= s_{k} t_{k+1}\left(\nabla f(y_k) + M^y_k\right) - t_{k+1} \beta_{k} \sqrt{s_{k}} \nabla f(x_k) ,
\end{align*}
and 
\begin{align*}
&(t_{k+1} - 1)(y_k - x_k) + (y_k - x^\star) - t_{k+1}(x_{k+1} - x_k) - (x_k - x^\star) \\
&= t_{k+1}(y_k - x_{k+1}) \\
&= s_{k} t_{k+1}\left(\nabla f(y_k) + M^y_k\right) .
\end{align*}
Thus \eqref{eq:a_k_inter} equivalently reads
\begin{align*}
E_{k+1} - E_{k} 
\leq& s_k\left(t_{k+1}^2 - t_{k+1} - t_{k}^2 \right)\left( f(x_k) - \min_{\Hc} f\right) \\
& + s_kt_{k+1} \left\langle \nabla f\left(y_{k}\right), s_{k} t_{k+1}\left(\nabla f(y_k) + M^y_k\right) - t_{k+1} \beta_k \sqrt{s_{k}} \nabla f(x_k) \right\rangle \\
& - \frac{s_{k}^2}{2} t_{k+1} \left(2t_{k+1} + 1\right) \left\|\nabla f\left(y_{k}\right)\right\|^{2} - \frac{s_{k}^2}{2} t_{k+1} \left(t_{k+1} - 1\right) \left\|\nabla f\left(x_{k}\right) - \nabla f\left(y_{k}\right)\right\|^{2} \\
&- s_kt_{k+1}\langle M_k , \nabla f(y_k)\rangle - \frac{1}{2}\left\|M_k\right\|^2 - \left\langle M_k , v_{k+1} \right\rangle \\
& + s_kt_{k+1}\left\langle M^y_k, s_{k} t_{k+1}\left(\nabla f(y_k) + M^y_k\right) \right\rangle\\
& - s_{k}^2 t_{k+1}^2 \left\langle M^y_k , \nabla f\left(y_{k}\right) \right\rangle - \frac{s_{k}^2}{2}t_{k+1}^2 \left\|M^y_k\right\|^{2} + \left\langle M_{k} , v_{k+1} \right\rangle .
\end{align*}
Simplifying, we get
\begin{equation}\label{eq:lyapuUtil}
\begin{aligned}
E_{k+1} - E_{k} 
&\leq s_k\left(t_{k+1}^2 - t_{k+1} - t_{k}^2 \right)\left( f(x_k) - \min_{\Hc} f\right)
- s_kt_{k+1} A_k \\
&- s_kt_{k+1}\langle M_k , \nabla f(y_k)\rangle - \frac{1}{2}\left\|M_k\right\|^2 \\
&+ s_k^2t_{k+1}^2\left\langle M^y_k , \nabla f(y_k) \right\rangle 
+ \frac{s_k^2}{2}t_{k+1}^2\left\| M^y_k \right\|^2 ,
\end{aligned}
\end{equation}
where
\[
A_k := t_{k+1}\beta_k\sqrt{s_k}\left\langle\nabla f(y_k),\nabla f(x_k)\right\rangle + \frac{s_{k}}{2} \left\|\nabla f\left(y_{k}\right)\right\|^{2} + \frac{s_{k}}{2} \left(t_{k+1} - 1\right) \left\|\nabla f\left(x_{k}\right) - \nabla f\left(y_{k}\right)\right\|^{2} .
\]
Replacing $M_k$ by its expression, developing and ignoring the negative terms, we arrive at
\begin{align*}
E_{k+1} - E_{k} 
&\leq s_k\left(t_{k+1}^2 - t_{k+1} - t_{k}^2 \right)\left( f(x_k) - \min_{\Hc} f\right)
- s_kt_{k+1} A_k \\
&- s_kt_{k+1}\left\langle  \left(\beta_{k}\sqrt{s_{k}}t_{k+1}M^x_k - \beta_{k-1}\sqrt{s_{k-1}}t_k M^x_{k-1}\right), \nabla f(y_k) \right\rangle \\
&- s_kt_{k+1}\left\langle  \left(\beta_{k}\sqrt{s_{k}}t_{k+1}M^x_k - \beta_{k-1}\sqrt{s_{k-1}}t_k M^x_{k-1}\right), M^y_k \right\rangle .
\end{align*}
Let us show that $A_k$ is nonnegative. Set $X=\nabla f\left(x_{k}\right)$ and  $Y=\nabla f\left(y_{k}\right)$. Thus 
\[
A_k = \frac{s_{k}}{2}(t_{k+1}-1) \left\|X\right\|^{2} + \frac{s_{k}}{2}t_{k+1} \left\|Y\right\|^{2} + \left(t_{k+1} \left(\beta_k\sqrt{s_{k}} - s_{k}\right) + s_{k} \right)\left\langle X,Y\right\rangle
\]
which is a quadratic form of $X$ and $Y$. For $\beta_k \equiv 0$, $A_k$ is obviously nonnegative. Let us now focus on $\beta_k \geq 0$. Since $t_k$ is of order $k$, Sylvester's criterion shows that for any $k \geq k_0$ with $k_0$ large enough (that depends on $\alpha$), $A_k$ is nonnegative if and only if
\[
\left(t_{k+1} \left(\beta_k\sqrt{s_{k}} - s_{k}\right) + s_{k} \right)^2 \leq s_{k}^2t_{k+1}\left(t_{k+1} - 1\right) ,
\]
or equivalently
\[
t_{k+1}^2\left((1-\beta_k/\sqrt{s_k})^2-1\right) \leq t_{k+1}\left(2(1-\beta_k/\sqrt{s_k}) -1 \right) - 1 .
\]
Using again that $t_k$ is of order $k$, a sufficient condition for this inequality to be satisfied for any $k \geq k_0$ (enlarging $k_0$ if necessary) is that 
\[
\left(s_k - \beta_k \sqrt{s_{k}}\right)^2 \leq s_{k}^2 \quad \text{and} \quad 2(1-\beta_k/\sqrt{s_k}) - 1 \geq 0 .
\]
Both hold true if $0 \leq \beta_k \leq \sqrt{s_{k}}/2$ as devised. Similar arguments entail that for $0\leq\epsilon_k \leq \beta_k(2\sqrt{s_{k}} - \beta_k) \leq s_k$ (such $\epsilon_k$ exists according to assumption $0 \leq \beta_k \leq \sqrt{s_{k}}/2$), we have
\[
A_k \geq \frac{\epsilon_k}{2} t_{k+1}\left\|\nabla f(y_k)\right\|^2 .
\]
Therefore
\begin{equation}\label{eq:lyapmainineq}
\begin{aligned}
E_{k+1} - E_{k} 
&\leq s_k\left(t_{k+1}^2 - t_{k+1} - t_{k}^2 \right)\left( f(x_k) - \min_{\Hc} f\right)
- \frac{\epsilon_k}{2}s_kt_{k+1}^2 \left\|\nabla f(y_k)\right\|^2 \\
&- s_kt_{k+1}\left\langle  \left(\beta_k\sqrt{s_{k}}t_{k+1}M^x_k - \beta_{k-1}\sqrt{s_{k-1}}t_k M^x_{k-1}\right), \nabla f(y_k) \right\rangle \\
&- s_kt_{k+1}\left\langle  \left(\beta_k\sqrt{s_{k}}t_{k+1}M^x_k - \beta_{k-1}\sqrt{s_{k-1}}t_k M^x_{k-1}\right), M^y_k \right\rangle .
\end{aligned}
\end{equation}
Applying Young's and Jensen's inequalities to the last two terms and rearragining then yields \eqref{eq:genrecursion}.
\end{proof}

\subsection{Convergence guarantees}


The main results of this section are summarized in the following theorem. 
\begin{theorem}\label{thm:vals_thm}
Let $f$ satisfy \eqref{eq:mainassump} and $(x_k)_{k \in \NN}$ be a sequence generated by Algorithm~\ref{alg:hdde}, with $\alpha\geq 3$, $\beta_k=\eta\frac{\sqrt{s_{k}}}{2}$, $\eta \in ]0,1]$, and $s_{k} \in ]0,1/L]$ is nonincreasing. Assume that
\begin{equation}
\sum_{k \in \NN} s_k k\|M^x_k\| < +\infty \quad \text{and} \quad \sum_{k \in \NN} s_k k\|M^y_k\| < +\infty . \quad \tag{$J_0$} \label{eq:J0}
\end{equation}
Then the following convergence rates hold on objective values and gradients:
\begin{enumerate}[label=(\roman*)]
    \item $f\left(x_{k}\right)-\min_{\Hc} f=\mathcal{O}\left(\frac{1}{s_k k^{2}}\right) \text{ as } k \rightarrow \infty$.\label{item:val_first_det}
    
    \item We have
    \[
    \sum_{k \in \NN} s_k k^2 \|\nabla f(y_k)\|^2 < +\infty .
    \]
    If $\inf_{k \in \NN} s_k > 0$, then
    \[
    \sum_{k \in \NN} k^2 \|\nabla f(x_k)\|^2 < +\infty \quad \text{and} \quad \sum_{k \in \NN} k^2 \|\nabla f(y_k)\|^2 < +\infty .
    \]
    \label{item:val_second_det}

    Suppose now that $\alpha > 3$. Then
    \item $\sum_{k \in \NN} s_k k(f\left(x_{k}\right)-\min_{\Hc} f) < +\infty$.\label{item:sumf}
    \item If $s_k \equiv s > 0$, then
    \label{item:vel}
    \begin{enumerate}
    \item $\sup_{k \in \NN} k\left\|x_{k}-x_{k-1}\right\|<+\infty$\label{item:vel_first}.
    \item $\sum_{k \in \NN} k\left\|x_{k}-x_{k-1}\right\|^{2}<+\infty$\label{item:vel_second}.
    \item $\left(x_{k}\right)_{k \in \NN}$ converges weakly to a minimizer of $f$ \label{item:iter}.
    \end{enumerate}
\end{enumerate}
\end{theorem}

\begin{remark}
The choice of a constant step-size is crucial for obtaining the rates on velocity, and hence the convergence of the sequence. See Section~\ref{subsec:role_step-size} for a detailed discussion.
\end{remark}

\begin{proof}
We start from \eqref{eq:genrecursion} in Lemma~\ref{lem:lyaperr}, taking $\epsilon_k=\beta_k(2\sqrt{s_k}-\beta_k^2)=\eta s_k(1-\eta/4)$, and ignoring the first negative term in the rhs of \eqref{eq:genrecursion} since $t_{k+1}^2 - t_{k+1} \leq t_{k}^2$ when $\alpha \geq 3$ (see Remark~\ref{rem:lemma}). We get for all $k \geq k_0$,
\begin{multline}\label{eq:genrecursion2}
E_{k+1} - E_{k} 
\leq - \frac{\eta(1-\eta/4)}{4}s_kt_{k+1}^2 \left\|\nabla f(y_k)\right\|^2 \\
+ \frac{s_k^2}{2}t_{k+1}^2\left\|M^y_k\right\|^2 
+ \frac{\eta}{4}s_k^2t_{k+1}^2\left(\frac{2}{1-\eta/4} + \eta \right)\left\|M^x_k\right\|^2 
+ \frac{\eta}{4}s_{k-1}^2t_{k}^2\left(\frac{2}{1-\eta/4} + \eta \right)\left\|M^x_{k-1}\right\|^2 .
\end{multline}
Iterating and dropping the negative term, we get, in view of the definition of $E_k$ and using $k_0$ as defined in Lemma~\ref{lem:lyaperr}
\begin{multline}\label{eq:genrecursioniter}
s_kt_k^2 (f(x_k)-\min_{\Hc} f)+\frac{1}{2} \|v_k\|^2 \leq s_{k_0}t_{k_0}^2 (f(x_{k_0})-\min_{\Hc} f)+\frac{1}{2} \|v_{k_0}\|^2 
- \sum\limits_{i=k_0}^k \langle v_{i+1}, M_{i}\rangle \\
+ \sum_{i=k_0}^k\frac{s_i^2}{2}t_{i+1}^2\left\|M^y_i\right\|^2 
+ \frac{\eta}{2}\left(\frac{2}{1-\eta/4} + \eta \right)\sum_{i=k_0-1}^ks_i^2t_{i+1}^2\left\|M^x_i\right\|^2 .
\end{multline}
Since $t_k$ is of the order of $k$, the last two series in \eqref{eq:genrecursioniter} converge, hence are bounded thanks to assumption~\eqref{eq:J0}. This implies that there exists a constant $C > 0$ such that
\[
\|v_k\|^2 \leq C + 2\sum\limits_{i=1}^{k+1} \|v_{i}\|\|M_{i-1}\|
\leq C + 2\sum\limits_{i=1}^{k+1} \|v_{i}\|\left(\eta s_it_i\|M^x_{i-1}\|+s_it_i\|M^y_{i-1}\|\right)
\]
where we dropped the nonnegative term $s_kt_k^2 (f(x_k)-\min_{\Hc} f)$ in \eqref{eq:genrecursioniter}. In view of assumption~\eqref{eq:J0}, the Gronwall-Bellman inequality in Lemma~\ref{lemma:A2} applies which gives
\begin{equation}\label{eq:vkbound}
C' := \sup_{k \geq k_0} \|v_k\| \leq \sqrt{C} + \sum_{k \in \NN} \left(\eta s_kt_k\|M^x_k\|+s_kt_k\|M^y_k\|\right) < +\infty .
\end{equation}
We then infer, in view of the definition of $E_k$, that
\[
st_k^2 (f(x_k)-\min_{\Hc} f)+\frac{1}{2} \|v_k\|^2 \leq C + \sum\limits_{i=k}^\infty \langle v_{i+1}, M_{i}\rangle \leq C+C'\sum\limits_{i=k}^{+\infty} \|M_i\| < +\infty, \qquad \forall k \geq k_0 ,
\]
This gives claim \ref{item:val_first_det} as $t_k=\frac{k-1}{\alpha-1}$.

\medskip

Summing \eqref{eq:genrecursion2} from $k_0$ to $k$ gives
\begin{multline*}
\frac{\eta(1-\eta/4)}{4}\sum_{i=k_0}^ks_it_{i+1}^2 \left\|\nabla f(y_i)\right\|^2 
\leq E_{k_0} - E_{k+1} \\
+ \sum_{i=k_0}^k\frac{s_i^2}{2}t_{i+1}^2\left\|M^y_i\right\|^2 
+ \frac{\eta}{2}\left(\frac{2}{1-\eta/4} + \eta \right)\sum_{i=k_0-1}^ks_i^2t_{i+1}^2\left\|M^x_i\right\|^2 .
\end{multline*}
The proof of \ref{item:val_first_det} has also shown that $E_k$ is bounded (in fact even converges thanks to Lemma~\ref{lemma:A3-det}). The conclusions of \ref{item:val_second_det} at $y_k$ follow by invoking the summability assumption~\eqref{eq:J0}. To prove those at $x_k$, just observe from Lipschitz continuity of $\nabla f$ and \eqref{eq:genalgstepx} that
\[
\norm{\nabla f(x_{k+1})}^2 \leq 2(1+Ls_k)^2\norm{\nabla f(y_k)}^2 + 2L^2s_k^2\norm{M^y_k}^2 .
\]

\medskip

To show \ref{item:sumf}, we first observe that $t_{k}^2 - t_{k+1}^2 + t_{k+1} \geq \frac{\alpha-3}{(\alpha-1)^2}k$ since $t_{k+1}=\frac{k}{\alpha-1}$. Now, returning to \eqref{eq:genrecursion} and discarding the gradient term, we get
\begin{multline*}
E_{k+1} - E_{k} \leq -\frac{\alpha-3}{(\alpha-1)^2}s_k k \left( f(x_k) - \min_{\Hc} f\right)\\
+ \frac{s_k^2}{2}t_{k+1}^2\left\|M^y_k\right\|^2 
+ \frac{\eta}{4}s_k^2t_{k+1}^2\left(\frac{2}{1-\eta/4} + \eta \right)\left\|M^x_k\right\|^2 
+ \frac{\eta}{4}s_{k-1}^2t_{k}^2\left(\frac{2}{1-\eta/4} + \eta \right)\left\|M^x_{k-1}\right\|^2 .
\end{multline*}
Arguing as in the proof of \ref{item:val_second_det}, we get the claim.

\medskip

Let us now prove the claims of \ref{item:vel}. Applying Lemma~\ref{lemma:A1} with $x = x_k$ and $y = y_k$, rearranging the inner product, and using that $x_{k+1} = y_k - s \nabla f_k(y_k)$, we get
\begin{align}\label{eq:descvel}
f_k\left(x_{k+1}\right) + \frac{1}{2s} \| x_{k+1}-x_k\|^2 &\leq f_k\left(x_k\right) + \frac{1}{2s} \| y_k - x_k \|^2. 
\end{align}
Denote
\begin{align*}
    \widetilde{\nabla f}(x_k) = \nabla f(x_k) + M^x_k, \qquad
    \widetilde{\nabla f}(x_{k-1}) = \nabla f(x_{k-1}) + M^x_{k-1} .
\end{align*}
Let us estimate the last term of \eqref{eq:descvel}. According to the definition of $y_k$ with constant step-size, i.e., $s_k\equiv s$, along with the Cauchy-Schwarz inequality, we obtain
\begin{align*}
\| y_k - x_k \|^2 
=& 
\left\| 
\alpha_k (x_k - x_{k-1}) 
-\beta \sqrt{s}\left(\widetilde{\nabla f}(x_k) - \widetilde{\nabla f}(x_{k-1})\right)-\frac{\beta \sqrt{s}}{k}\widetilde{\nabla f}(x_{k-1})
\right\|^2\\
\leq& \alpha_k^2 \| x_k - x_{k-1} \|^2 + \beta^2 s \| \widetilde{\nabla f}(x_k) - \widetilde{\nabla f}(x_{k-1}) \|^2 + \frac{\beta^2 s}{k^2}\| \widetilde{\nabla f}(x_{k-1})\|^2 \\
&- \frac{2 \beta \alpha_k \sqrt{s}}{k} \langle x_k - x_{k-1},  \widetilde{\nabla f}(x_{k-1}) \rangle + \frac{2 \beta^2 {s}}{k} \langle \widetilde{\nabla f}(x_k)  - \widetilde{\nabla f}(x_{k-1}),  \widetilde{\nabla f}(x_{k-1})  \rangle \\
&- 2\beta \alpha_{k}\sqrt{s} \langle x_k - x_{k-1},  \widetilde{\nabla f}(x_k) - \widetilde{\nabla f}(x_{k-1})\rangle\\
\leq& \alpha_k^2 \| x_k - x_{k-1} \|^2 + \beta^2 s \| \widetilde{\nabla f}(x_k) - \widetilde{\nabla f}(x_{k-1}) \|^2 + \frac{\beta^2 s}{k^2}\| \widetilde{\nabla f}(x_{k-1}) \|^2 \\
&+ \frac{2 \beta \alpha_k \sqrt{s}}{k} \|x_k - x_{k-1}\|  \|\widetilde{\nabla f}(x_{k-1})\| + \frac{2 \beta^2 s}{k} \| \widetilde{\nabla f}(x_k) - \widetilde{\nabla f}(x_{k-1})\| \| \widetilde{\nabla f}(x_{k-1}) \| \\
&- 2\beta \alpha_{k} \sqrt{s} \langle x_k - x_{k-1},  \nabla f(x_k)  -  \nabla f(x_{k-1})\rangle + 2\beta \alpha_{k}\sqrt{s}\| x_k - x_{k-1}\|  \| M^x_k -  M^x_{k-1} \| \\
\leq& \alpha_k^2 \| x_k - x_{k-1} \|^2 + \beta^2 s \| \widetilde{\nabla f}(x_k) - \widetilde{\nabla f}(x_{k-1}) \|^2 + \frac{\beta^2 s}{k^2}\| \widetilde{\nabla f}(x_{k-1}) \|^2 \\
&+ \frac{2 \beta \alpha_k \sqrt{s}}{k} \|x_k - x_{k-1}\|  \|\widetilde{\nabla f}(x_{k-1})\| + \frac{2 \beta^2 s}{k} \| \widetilde{\nabla f}(x_k) - \widetilde{\nabla f}(x_{k-1})\| \| \widetilde{\nabla f}(x_{k-1}) \| \\
&+ 2\beta \alpha_{k}\sqrt{s}\| x_k - x_{k-1}\|  \| M^x_k -  M^x_{k-1} \| ,
\end{align*}
where we used monotonicity of $\nabla f$.
To lighten notation, let $g_k := \sqrt{s}\left(\|\widetilde{\nabla f}(x_k)\| + \| \widetilde{\nabla f}(x_{k-1})\|\right)$ and $e_k = \sqrt{s} \left(\| M^x_k - M_{k-1}^x \|\right)$. The last inequality now reads
\begin{align*}
\| y_k - x_k \|^2 & \leq  \alpha_k^2 \| x_k - x_{k-1} \|^2 + \beta^2  g_k^2 + \frac{\beta^2 }{k^2} g_k^2  + \frac{2 \beta \alpha_{k}}{k} \| x_k - x_{k-1} \| g_k + \frac{2 \beta^2 }{k}g_k^2  + 2 \beta \alpha_{k}\| x_k - x_{k-1}\|  e_k .
\end{align*}
Plugging this into \eqref{eq:descvel}, we obtain
\begin{multline*}
f_k(x_{k+1}) + \frac{1}{2s} \| x_{k+1} - x_k \|^2 
\leq  f_k(x_k) +  \frac{\alpha_k^2}{2s} \| x_k - x_{k-1} \|^2 + \frac{\beta^2}{2s}  g_k^2 + \frac{\beta^2}{2sk^2}g_k^2\\
+ \frac{\beta \alpha_k}{sk}\| x_k - x_{k-1} \| g_k
 + \frac{\beta^2}{sk}g_k^2  + \frac{\alpha_k \beta}{s} \| x_k - x_{k-1}\|  e_k .
\end{multline*}
Using the definition of $f_k$, we can write:
\begin{align*}
f(x_{k+1}) + \langle M^y_k,x_{k+1}\rangle + \frac{1}{2s} \| x_{k+1} - x_k \|^2 &\leq  f(x_k) + \langle M^y_k,x_{k}\rangle + \frac{1}{2s}\alpha_k^2 \| x_k - x_{k-1} \|^2 + \frac{1}{2s}\beta^2  g_k^2 + \frac{1}{2s}\frac{\beta^2g_k^2}{k^2}\\
&\quad+ \frac{\beta \alpha_k}{k}\| x_k - x_{k-1} \| g_k
 + \frac{\beta^2 g_k^2}{k s}  + \frac{\alpha_k \beta}{s} \| x_k - x_{k-1}\|  e_k .
\end{align*}
Multiplying by $s$, and using Cauchy-Schwarz inequality 
we get
\begin{align*}
s\left(f(x_{k+1}) - \min_{\Hc} f\right)  + \frac{1}{2} \| x_{k+1} - x_k \|^2 &\leq  s\left(f(x_k) -\min_{\Hc} f\right) + s\| M^y_k\| \|x_{k+1} - x_{k}\| + \frac{1}{2}\alpha_k^2 \| x_k - x_{k-1} \|^2 + \frac{1}{2}\beta^2  g_k^2 \\
&\quad+\frac{1}{2}\frac{\beta^2g_k^2}{k^2}+ \frac{\beta \alpha_k s}{k}\| x_k - x_{k-1} \| g_k
 + \frac{\beta^2 g_k^2}{k }  + \alpha_k \beta \| x_k - x_{k-1}\|  e_k .
\end{align*}
Let \(\theta_k = s\left(f(x_k) - \min_{\Hc} f\right)\) and \(d_k = \frac{1}{2} \| x_k - x_{k-1} \|^2\). The last inequality now reads
\begin{align*}
\theta_{k+1} + d_{k+1}
&\leq \theta_k + \frac{(k-\alpha)^2}{k^2} d_k + \frac{\beta^2 g_k^2}{2}\left(1 + \frac{1}{k}\right)^2 + \|x_k - x_{k-1}\| \left(\frac{ \beta \alpha_k }{k}g_k + \beta \alpha_k  e_k \right) + s\| M^y_k\| \|x_{k+1} - x_{k}\|\\
&\leq \theta_k + \frac{(k-\alpha)^2+1}{k^2} d_k + \frac{\beta^2 g_k^2}{2}\left(\left(1 + \frac{1}{k}\right)^2+\alpha_k^2\right) + \|x_k - x_{k-1}\| \beta \alpha_k  e_k + s\| M^y_k\| \|x_{k+1} - x_{k}\| .
\end{align*}
Multiplying by $k^2$ and rearranging, we get
\begin{equation}\label{eq:recursive_vel}
\begin{aligned}
k^2 \theta_{k+1}+k^2 d_{k+1}+\left(2(\alpha-1)k-\alpha^2\right) d_k &\leq (k-1)^2 \theta_k + (k-1)^2 d_k + \left(2k-1\right)\theta_k \\
&\quad + \frac{\beta^2}{2}(1+\alpha_k^2)(k+1)^2g_k^2 + k\|x_k - x_{k-1}\| \left( k \beta \alpha_k e_k\right) \\ &\quad + k\|x_{k+1} - x_{k}\| (sk\| M^y_k\| ).
\end{aligned}
\end{equation}
Let us denote for short $m_k := \beta k  \alpha_k e_k$. 
The term $(2(\alpha-1)k-\alpha^2) d_k$ is non-negative for $k \geq k_0$ large enough since $\alpha > 1$. We can then discard it from the lhs of \eqref{eq:recursive_vel}. Moreover $0 \leq \alpha_k \leq 1$.
Thus iterating from \(k = k_0\) to \(k = K > k_0\) we get
\begin{align}
K^2 d_{K+1} 
&\leq (k_0 - 1)^2 \theta_{k_0} + (k_0 - 1)^2 d_{k_0} + 2 \sum\limits_{k=k_0}^K k \theta_k +  \beta^2 \sum\limits_{k=k_0}^K (k+1)^2 g_k^2 \nonumber\\
&+ \sum\limits_{k=k_0}^K m_k k\|x_k - x_{k-1}\| + 
s^2\sum\limits_{k=k_0}^{K+1} ((k-1)\norm{M^y_{k-1}}) (k-1)\|x_{k} - x_{k-1}\| \nonumber\\
&= (k_0 - 1)^2 \theta_{k_0} + (k_0 - 1)^2 d_{k_0} + 2 \sum\limits_{k=k_0}^K k \theta_k +  \beta^2 \sum\limits_{k=k_0}^K (k+1)^2 g_k^2 \nonumber\\
&+ \sum\limits_{k=k_0}^K (m_k+s^2k\norm{M^y_{k-1}}) k\|x_k - x_{k-1}\| + 
(s^2K\norm{M^y_k}) K\|x_{K+1} - x_{K}\| \nonumber\\
&\leq (k_0 - 1)^2 \theta_{k_0} + (k_0 - 1)^2 d_{k_0} + 2 \sum\limits_{k=k_0}^K k \theta_k +  \beta^2 \sum\limits_{k=k_0}^K (k+1)^2 g_k^2 \nonumber\\
&+ \sum\limits_{k=k_0}^K (m_k+s^2k\norm{M^y_{k-1}}) k\|x_k - x_{k-1}\| 
+ s^4K^2\norm{M^y_k}^2 + \frac{K^2}{2}d_{K+1} . \label{eq:dkineq}
\end{align}
From claims \ref{item:val_first_det}, \ref{item:val_second_det} and \ref{item:sumf} proved above as well as assumption~\eqref{eq:J0}, we have
\begin{equation}\label{eq:sum_res_{k}}
\sum_{k \in \NN} k \theta_k < +\infty \quad \text{and} \quad \sum_{k \in \NN} k^2 g_k^2 < +\infty.
\end{equation}
Thus there exists a constant \(C > 0\) such that
\begin{align*}
(K \| x_{K+1} - x_K \|)^2 \leq & 4C + \sum\limits_{k=1}^K 4(m_k+s^2k\norm{M^y_{k-1}})\left(k\|  x_k - x_{k-1} \|\right),
\end{align*}
Using assumption~\eqref{eq:J0}, we can claim that
\begin{equation}\label{eq:sum_mk}
\sum_{k \in \NN} (m_k+s^2k\norm{M^y_{k-1}}) < +\infty .
\end{equation}
Applying Lemma~\ref{lemma:A2}, we obtain
\begin{align*}
\sup_{k \in \NN} k \| x_k - x_{k-1} \| < +\infty.
\end{align*}
To see that assertion \ref{item:vel_second} holds, let us return to \eqref{eq:recursive_vel} and this time without discarding the nonnegative term $\left(2(\alpha-1)k - \alpha^2\right)d_k$. Summing \eqref{eq:recursive_vel} gives
\begin{align*}
\sum_{k=k_0}^K \left(2(\alpha-1)k - \alpha^2\right)d_k 
&\leq (k_0 - 1)^2 \theta_{k_0} + (k_0 - 1)^2 d_{k_0} +
2 \sum\limits_{k=k_0}^K k \theta_k +  \beta^2 \sum\limits_{k=k_0}^K (k+1)^2 g_k^2 \\
&\quad + \sum\limits_{k=k_0}^K m_k k\|x_k - x_{k-1}\| + s^2\sum\limits_{k=k_0}^{K} (k\norm{M^y_k}) (k+1)\|x_{k+1} - x_{k}\| .
\end{align*}
In view of claim \ref{item:vel_first}, there exists a constant $C > 0$ such that
\begin{align*}
\left(\alpha -1\right)\sum_{k=k_0}^K \left(2k - \alpha -1\right)d_k 
&\leq (k_0 - 1)^2 \theta_{k_0} + (k_0 - 1)^2 d_{k_0} +
2 \sum\limits_{k=k_0}^K k \theta_k + \beta^2 \sum\limits_{k=k_0}^K (k+1)^2 g_k^2  \\
&+ C\sum\limits_{k=k_0}^K (m_k + s^2 k \norm{M^y_k}) .
\end{align*}
Passing to the limit as $K \to +\infty$, using \eqref{eq:sum_res_{k}} and \eqref{eq:sum_mk} gives claim \ref{item:vel_second}.\\


The proof for \ref{item:iter} is based on Opial's lemma (Lemma~\ref{lemma:A4}) with \( S = \argmin_{\Hc} f \). From \ref{item:val_first_det}, we have \( f(x_k) \to \min_{\Hc} f \). Moreover, recall that $v_k := (x_{k-1}-x^\star)+t_k(x_k-x_{k-1}+\beta\sqrt{s}\nabla f(x_{k-1}))$ is bounded, as shown in \eqref{eq:vkbound}. Claim~\ref{item:vel_first} also shows that $t_k(x_k-x_{k-1})$ is bounded and so is $t_k\nabla f(x_{k-1})$ thanks to \ref{item:val_second_det} (it is even $o(1)$). We then deduce that $(x_k)_{k \in \NN}$ is also bounded. Thus, for any weak subsequential cluster point of $x_k$, say $x_{k_j} \to \bar{x}$ as $j \to \infty$, the continuity of $f$ yields $f(\bar{x})=\lim_{j \to \infty}f(x_{k_j}) = \min_{\Hc} f$, that is, $\bar{x} \in S$. This shows that item (i) of Lemma~\ref{lemma:A4} holds. It remains to verify item (ii) of this lemma, that is, that \( \lim_{k \to \infty} \|x_k - x^\star\| \) exists for any \( x^\star \in S \). 

Denote for short the anchor sequence \( h_k = \frac{1}{2} \|x_k - x^\star\|^2 \). The idea of the proof is to establish a discrete second-order differential inequality satisfied by the sequence \( (h_k)_{k \in \NN} \). We recall the three-point identity
\begin{align*}
\frac{1}{2}\|a - b\|^2 + \frac{1}{2}\|a - c\|^2 = & \frac{1}{2}\|b - c\|^2 + \langle a - b, a - c \rangle,
\end{align*}
which holds for any \( a, b, c \in \Hc \). Applying this identity with \( b = x^\star, a = x_{k+1}, c = x_k \), we obtain
\begin{equation}\label{eq:three_point_hk}
\begin{aligned}
h_k - h_{k+1} = & \frac{1}{2} \|x_{k+1} - x_k\|^2 + \langle x_{k+1} - x^\star, x_k - x_{k+1} \rangle.
\end{aligned}
\end{equation}
By the definition of \( y_k \), we have:
\begin{align*}
x_k - x_{k+1} &=  y_k - x_{k+1} - \alpha_k \left(x_k - x_{k+1}\right) +  \beta \sqrt{s}\bigg( \left(\nabla f(x_k) + M^x_k\right) - \left(\nabla f(x_{k-1}) - M_{k-1}^x\right)\bigg) \\
\quad & + \frac{\beta \sqrt{s}}{k}\left( \nabla f(x_{k-1}) + M^x_{k-1}\right) .
\end{align*}
Thus
\begin{align*}
h_k - h_{k+1} = & \frac{1}{2}\|x_{k+1} - x_k\|^2 + \langle x_{k+1} - x^\star, y_k - x_{k+1} \rangle - \alpha_k \langle x_{k+1} - x^\star, x_k - x_{k+1} \rangle \\
\quad& + \langle x_{k+1} - x^\star, \beta\sqrt{s}\left(   \nabla f(x_k) - \nabla f(x_{k-1})\right)\rangle + \frac{\beta \sqrt{s}}{k} \langle x_{k+1} - x^\star,\nabla f(x_{k-1})  \rangle\\
\quad&+ \langle x_{k+1} - x^\star, \beta\sqrt{s}\left(   M^x_k - M^x_{k-1}\right)\rangle + \frac{\beta \sqrt{s}}{k} \langle x_{k+1} - x^\star,M^x_{k-1}  \rangle .
\end{align*}
Rearranging the equation, we can write
\begin{equation}\label{eq:h_k_recursive}
\begin{aligned}
- \langle x_{k+1} - x^\star, y_k - x_{k+1} \rangle = & h_{k+1} - h_{k} + \frac{1}{2}\|x_{k+1} - x_k\|^2  - \alpha_k \langle x_{k+1} - x^\star, x_k - x_{k+1} \rangle \\
\quad& + \langle x_{k+1} - x^\star, \beta\sqrt{s}\left(   \nabla f(x_k) - \nabla f(x_{k-1})\right)\rangle + \frac{\beta \sqrt{s}}{k} \langle x_{k+1} - x^\star,\nabla f(x_{k-1})  \rangle\\
\quad&+ \langle x_{k+1} - x^\star, \beta\sqrt{s}\left(   M^x_k - M^x_{k-1}\right)\rangle + \frac{\beta \sqrt{s}}{k} \langle x_{k+1} - x^\star,M^x_{k-1}  \rangle .
\end{aligned}
\end{equation}
Let us rewrite the rhs using the definition of $x_{k+1}$ from Algorithm~\ref{alg:hdde}
\begin{align*}
    -\langle x_{k+1} - x^\star, y_k - x_{k+1}\rangle &=
    -s\langle x_{k+1} - x^\star, \nabla f(y_k) + M^y_k\rangle\\
    &=-s\langle y_{k} -s\nabla f(y_k) - sM^y_k - x^\star, \nabla f(y_k)\rangle - s\langle x_{k+1} - x^\star, M^y_k\rangle \\
    &=-s\langle y_k - x^\star, \nabla f(y_k) \rangle + s^2\|\nabla f(y_k)\|^2 + s^2\langle M^y_k,\nabla f(y_k)\rangle - s\langle x_{k+1} - x^\star, M^y_k\rangle .
\end{align*}
Since $x^\star \in S$, $-s\langle y_k - x^\star, \nabla f(y_k) \rangle \le 0$ due to the monotonicity of $\nabla f$, hence we can write
\begin{align*}
-\langle x_{k+1} - x^\star, y_k - x_{k+1}\rangle &\le  s^2\|\nabla f(y_k)\|^2 + s^2\langle M^y_k,\nabla f(y_k)\rangle - s\langle x_{k+1} - x^\star, M^y_k\rangle .
\end{align*}
Therefore, plugging this relation into \eqref{eq:h_k_recursive}, we get the inequality
\begin{equation}\label{eq:h_k_ineq}
\begin{aligned}
&h_{k+1} - h_k + \frac{1}{2} \|x_{k+1} - x_k\|^2 
- \alpha_k \langle x_{k+1} - x^\star, x_k - x_{k-1} \rangle
+ \langle x_{k+1} - x^\star, \beta \sqrt{s} (\nabla f(x_k) - \nabla f(x_{k-1})) \rangle \\
&\quad+ \frac{\beta \sqrt{s}}{k} \left\langle x_{k+1} - x^\star, \nabla f(x_{k-1}) \right\rangle + \langle x_{k+1} - x^\star, \beta \sqrt{s}\left(   M^x_k - M^x_{k-1}\right)\rangle\\
&\quad+ \frac{\beta \sqrt{s}}{k} \langle x_{k+1} - x^\star,M^x_{k-1}  \rangle - s^2\|\nabla f(y_k)\|^2 - s^2\langle M^y_k,\nabla f(y_k)\rangle + s\langle x_{k+1} - x^\star, M^y_k\rangle \leq 0.
\end{aligned}
\end{equation}
Observe from \eqref{eq:three_point_hk} that
\begin{equation}\label{eq:h_k_recursive_prev}
\begin{array}{rl}
h_{k-1} - h_k = & \frac{1}{2} \|x_k - x_{k-1}\|^2 - \langle x_k - x^\star, x_k - x_{k-1} \rangle.
\end{array}
\end{equation}
Combining \eqref{eq:h_k_ineq} and \eqref{eq:h_k_recursive_prev}, we get
\begin{align*}
h_{k+1} - h_k - \alpha_k (h_k - h_{k-1}) &\leq  -\frac{1}{2}\|x_{k+1} - x_k\|^2 + \alpha_k \left( \frac{1}{2}\|x_k - x_{k-1}\|^2 + \langle x_k - x_{k-1}, x_{k+1} - x_k \rangle \right) \\
& - \beta \sqrt{s}  \langle x_{k+1} - x^\star, \nabla f(x_k) - \nabla f(x_{k-1})\rangle - \frac{\beta \sqrt{s}}{k}\langle x_{k+1} - x^\star ,\nabla f(x_{k-1}) \rangle \\
&- \beta \sqrt{s}\langle x_{k+1} - x^\star,  M^x_k - M^x_{k-1}\rangle- \frac{\beta \sqrt{s}}{k} \langle x_{k+1} - x^\star,M^x_{k-1}  \rangle\\
& + s^2\|\nabla f(y_k)\|^2 + s^2\langle M^y_k,\nabla f(y_k)\rangle - 
 s \langle x_{k+1} - x^\star, M^y_k \rangle .
\end{align*}
Let us now group all the terms directly involving the noise terms as 
\begin{align*}
\varepsilon_k := |\langle \beta\sqrt{s}\left( M^x_k - M^x_{k-1} + k^{-1}M^x_{k-1}\right) - sM^y_k, x_{k+1} - x^\star\rangle + s^2 \langle M^y_k, \nabla f(y_k) \rangle| .
\end{align*}
So, the last inequality above now reads
\begin{equation}\label{eq:h_k_compiled}
\begin{aligned}
h_{k+1} - h_k - \alpha_k (h_k - h_{k-1}) &\leq  -\frac{1}{2}\|x_{k+1} - x_k\|^2 + \alpha_k \left( \frac{1}{2}\|x_k - x_{k-1}\|^2 + \langle x_k - x_{k-1}, x_{k+1} - x_k \rangle \right) \\
& - \beta \sqrt{s} \langle x_{k+1} - x^\star, \nabla f(x_k) - \nabla f(x_{k-1})\rangle - \frac{\beta \sqrt{s}}{k}\langle x_{k+1} - x^\star ,\nabla f(x_{k-1}) \rangle \\
& + s^2 \|\nabla f(y_k)\|^2 + \varepsilon_k .
\end{aligned}
\end{equation}
Let us rewrite first inner product on the rhs of \eqref{eq:h_k_compiled} as
\begin{align*}
\left\langle x_{k+1}-x^\star, \nabla f\left(x_{k}\right)- \nabla f\left(x_{k-1}\right)\right\rangle 
& = \left\langle x_{k} - x^\star, \nabla f\left(x_{k}\right)\right\rangle - \left\langle x_{k-1} - x^\star, \nabla f\left(x_{k-1}\right)\right\rangle \\ & + \left\langle x_{k+1} - x_{k}, \nabla f\left(x_{k}\right)\right\rangle - \left\langle x_{k+1} - x_{k-1}, \nabla f\left(x_{k-1}\right)\right\rangle .
\end{align*} 
Inserting this into \eqref{eq:h_k_compiled}, we obtain
\begin{equation}\label{eq:h_k_det}
\begin{aligned}
h_{k+1} - h_k - \alpha_k (h_k - h_{k-1}) &\leq  -\frac{1}{2}\|x_{k+1} - x_k\|^2 + \alpha_k \left( \frac{1}{2}\|x_k - x_{k-1}\|^2 + \langle x_k - x_{k-1}, x_{k+1} - x_k \rangle \right)\\
&\quad -\beta \sqrt{s}\left\langle x_{k} - x^\star, \nabla f\left(x_{k}\right)\right\rangle +\beta\sqrt{s} \left\langle x_{k-1} - x^\star, \nabla f\left(x_{k-1}\right)\right\rangle\\ 
& \quad - \beta\sqrt{s}\langle x_{k+1} - x_{k}, \nabla f(x_k) \rangle + \beta\sqrt{s} \langle x_{k+1} - x_{k-1}, \nabla f(x_{k-1})\rangle\\
 & \quad - \frac{\beta\sqrt{s}}{k}\langle x_{k+1} - x^\star, \nabla f(x_{k-1})\rangle + s^2 \|\nabla f(y_k)\|^2 + \varepsilon_k .
\end{aligned}
\end{equation}
Rearranging the terms in the expression can yield
\begin{equation}\label{eq:h_k_inter}
\begin{aligned}
h_{k+1} &- h_k - \alpha_k \left(h_k - h_{k-1}\right) 
+ \beta \sqrt{s} \langle x_k - x^\star, \nabla f(x_k) \rangle -  \beta \sqrt{s} \langle x_{k-1} - x^\star, \nabla f(x_{k-1}) \rangle\\
&  \leq -\frac{1}{2} \|x_{k+1} - x_k\|^2 + \alpha_k \left( \frac{1}{2} \|x_k - x_{k-1}\|^2 + \langle x_k - x_{k-1}, x_{k+1} - x_k \rangle \right) \\
& \quad - \beta\sqrt{s}\langle x_{k+1} - x_{k}, \nabla f(x_k) \rangle + \beta\sqrt{s} \langle x_{k+1} - x_{k-1}, \nabla f(x_{k-1})\rangle\\
 & \quad - \frac{\beta\sqrt{s}}{k}\langle x_{k+1} - x^\star, \nabla f(x_{k-1})\rangle + s^2 \|\nabla f(y_k)\|^2 + \varepsilon_k .
\end{aligned}
\end{equation}
The motivation behind the last rearrangement will be made clear later when we reason that all the terms on the rhs are part of a summable sequence. Set $\theta_k := h_k - h_{k-1} + \beta\sqrt{s}\langle x_{k-1} - x^\star,  \nabla f(x_{k-1})\rangle$. From \eqref{eq:h_k_inter}, and using the relation $\alpha_k = \left(1 - \frac{\alpha}{k}\right)$ we infer that
\begin{align*}
    \theta_{k+1} - \alpha_k\theta_k &\le \alpha\frac{\beta\sqrt{s}}{k}\langle x_{k-1} -x^\star, \nabla f(x_{k-1})\rangle -\frac{1}{2} \|x_{k+1} - x_k\|^2 + \alpha_k \left( \frac{1}{2} \|x_k - x_{k-1}\|^2 + \langle x_k - x_{k-1}, x_{k+1} - x_k \rangle \right) \\
& \quad - \beta\sqrt{s}\langle x_{k+1} - x_{k}, \nabla f(x_k) \rangle + \beta\sqrt{s} \langle x_{k+1} - x_{k-1}, \nabla f(x_{k-1})\rangle\\
 & \quad - \frac{\beta\sqrt{s}}{k}\langle x_{k+1} - x^\star, \nabla f(x_{k-1})\rangle + s^2 \|\nabla f(y_k)\|^2 + \varepsilon_k .
\end{align*}
Applying Young's inequality to each inner product in the rhs, ignoring the square negative term and rearranging, we get
\begin{equation}\label{eq:h_k_final}
\begin{aligned}
    \theta_{k+1} - \alpha_k\theta_k 
    &\leq \frac{(\alpha^2+1)\beta^2C^2s}{2k^3} + \frac{3}{2}k\|\nabla f(x_{k-1})\|^2 \\
    &+ \|x_k - x_{k-1}\|^2 + \frac{(1+\beta^2s)}{2} \|x_{k+1} - x_k \|^2 \\
    &+\frac{1}{2}\|\nabla f(x_k) \|^2 + \frac{\beta^2s}{2} \| x_{k+1} - x_{k-1}\|^2 + s^2 \|\nabla f(y_k)\|^2 + \varepsilon_k .
\end{aligned}
\end{equation}
where we have used that $\alpha_k \leq 1$ and $\sup_{k \in \NN} \norm{x_k-x^\star} \leq C$ for some constant $C > 0$ as proved above.
Gathering the rhs of \eqref{eq:h_k_final} in a nonnegative summable single term $e_k$,  we can write \eqref{eq:h_k_final} as
\begin{align*}
\theta_{k+1} \leq \alpha_k\theta_k + e_k .
\end{align*}
Taking the positive part (following same line of reasoning as in \cite{ACPR}), we can write
\begin{align*}
\left[\theta_{k+1}\right]_{+} \leq \alpha_k\left[\theta_k\right]_{+} + e_k .
\end{align*}
At this point, we can safely argue that $\sum_{k\in\NN}ke_k < +\infty$. In fact, for the terms involving the gradients and discrete velocities, this is due to \ref{item:val_second_det} and \ref{item:vel_second}. We also have $\sum_{k\in\NN}k\varepsilon_k < +\infty$ due to the boundedness of the sequence $(x_k)_{k \in \NN}$ and the assumption~\eqref{eq:J0}. We are then in position to invoke Lemma~\ref{lemma:A6} to deduce that
\begin{align*}
\sum_{k \in \NN} \big[ h_k - h_{k-1} + \beta \sqrt{s} \langle x_{k-1} - x^\star, \nabla f(x_{k-1}) \rangle \big]_+ < +\infty.
\end{align*}
The sequence \( (\|\nabla f(x_{k})\|)_{k \in \NN} \) is also summable. To see this we have 
\[
\|\nabla f(x_k)\| \leq \frac{1}{2k^2} + \frac{1}{2}k^2\|\nabla f(x_k)\|^2.
\]
The first term is summable and that of the second follows from \ref{item:val_second_det}. We then deduce that \( \sum_k [h_k - h_{k-1}]_+ < +\infty \), which implies that the limit of the sequence \( (h_k)_{k \in \NN} \) exists. This completes the proof.
\end{proof}

\paragraph{From $\mathcal{O}(\cdot)$ to $o(\cdot)$ rates}
The $\mathcal{O}(1/k^2)$ convergence rate on values and
velocities be improved to the even faster asymptotic rate $o(1/k^2)$.
\begin{theorem}\label{thm:fast_thm}
Let $f$ satisfy \eqref{eq:mainassump} and $(x_k)_{k \in \NN}$ be a sequence generated by Algorithm~\ref{alg:hdde}, with $\alpha > 3$, a constant step-size $s_k \equiv s > 0$, and assume that \eqref{eq:J0} holds. Then 
\begin{enumerate}[label=(\roman*)]
\item $f(x_k) - \min_{\Hc} f = o\left(\frac{1}{k^2}\right)$ as $k \rightarrow +\infty$ and \label{item:fast_first}
\item $\|x_k - x_{k-1}\| = o\left(\frac{1}{k}\right)$ as $k \rightarrow +\infty$.  \label{item:fast_second}
\end{enumerate}
\end{theorem}

\begin{proof}
Let us start from \eqref{eq:recursive_vel} and recall the notation there. Let us define:
$W_{k} :=(k-1)^{2} d_k+(k-1)^{2} \theta_{k}$, where we have for $k \geq k_0$ large enough that
\begin{align*}
W_{k+1} \leq W_k + q_k, \quad \text{where} \quad
q_k = \left(2k-1\right)\theta_k + \frac{\beta^2}{2}(1+\alpha_k^2)(k+1)^2g_k^2 + k\|x_k - x_{k-1}\| m_k .
\end{align*}
We have already shown in the proof of Theorem~\ref{thm:vals_thm}\ref{item:vel} that $\sum_{k \in \NN} q_k <+\infty$. It then follows from Lemma~\ref{lemma:A3-det} that $W_{k}$ converges as $k \rightarrow+\infty$. Theorem~\ref{thm:vals_thm}\ref{item:sumf} and \ref{item:vel} also ensure that $(W_k/k)_{k \in \NN}$ is summable. The claim then follows from Lemma~\ref{lemma:Asumo}. 
\end{proof}

\section{Stochastic Errors}\label{s:stoch}
    Our goal now is to address the case where the errors in the gradient computation become stochastic. Note that in the deterministic case of the previous section, we did not exploit any particular structure on the noise, necessitating meticulous handling of numerous cross-terms. In this section, we will focus on objectives of the form \eqref{eq:minExpecobj} and leverage random sampling to get conditional unbiasedness of gradient estimates, which will allow us to get rid of many of these cross-terms involving the noise. This assumption is reasonable and widely accepted in the stochastic optimization literature. Of course, the stochastic setting brings other challenges that must be addressed properly. We will not only provide convergence rates in expectation, which is the dominant approach in the literature, but we will also establish a.s. rates as well as other a.s. convergence guarantees.

\medskip
    
    We denote by $\left(\Omega, \Fc, \PP\right)$ our probability space. We assume that
    \begin{align}\label{eq:poprisk}
    \nabla f(x) = \EE_{\zeta \sim \mu}[\nabla F(x,\zeta)] ,
    \end{align}
    where 
    \begin{itemize}
    \item $F(x,\cdot)$ is $\mu$-integrable for each $x \in \Hc$; 
    \item $F(\cdot,\xi)$ is continuously differentiable and convex for each $\xi$;
    \item $\exists l: \Xi \to \RR^+$ such that
    \begin{equation}\label{eq:assumpoprisk}     
    \EE_{\zeta \sim \mu}[l(\zeta)^2] < +\infty \quad \text{and} \quad \text{for every $\zeta \in \Xi$}, ~~ \norm{\nabla F(x,\zeta)-\nabla F(z,\zeta)} \leq l(\zeta) \norm{x-z}, \quad \forall x, z \in \Hc .
	\end{equation}
	\end{itemize}
    This assumption implies that $f$ is convex and $\nabla f$ is well-defined and $L$-Lipschitz continuous with 
    \[
    L = \int_{\Xi} l(\zeta)d\mu(\zeta) \leq \EE_{\zeta \sim \mu}[l(\zeta)^2]^{1/2} < +\infty .
    \]
    In our setting, we consider the setting where the randomness at each iteration is introduced through the stochastic gradient estimates $G^{x^-}_k$, $G^x_k$ and $G^y_k$ of $\nabla f(x_{k-1})$, $\nabla f(x_k)$ and $\nabla f(y_k)$ respectively. We denote the corresponding events as $\omega^{x^-}_k$, $\omega^x_k$ and $\omega^y_k$:
    \begin{itemize}
    \item Event $\omega^x_k$ corresponds to sampling $(\zeta^x_i)_{i \in [N^x_k]} \sim_{iid} \mu$, and taking
    \[
    G^x_ k = \frac{1}{N^x_k} \sum_{i=1}^{N^x_k} \nabla F(x_k,\zeta^x_i).
    \]
    \item Similarly, event $\omega^{x^-}_k$ corresponds to sampling $(\zeta^{x^-}_i)_{i\in [N^{x^-}_k]}\sim_{iid}\mu$, and computing
    \[
    G^{x^-}_k = \frac{1}{N^{x^-}_k} \sum_{i=1}^{N^{x^-}_k} \nabla F(x_{k-1},\zeta^{x^-}_i) .  
    \]
    \item Event $\omega^y_k$ corresponds to sampling $(\zeta^y_i)_{i \in [N_k]} \sim_{iid} \mu$, and taking
    \[
    G^y_k = \frac{1}{N^y_k} \sum_{i=1}^{N^y_k} \nabla F(y_k,\zeta^y_i) .
    \]
     \end{itemize}
    We will use the shorthand notation for the events $\omega^{x^-,x}_k=(\omega^{x^-}_k,\omega^x_k)$ and $\omega_k = (\omega^{x^-,x}_k,\omega^y_k)$. We also define the stochastic noise components as
    \begin{equation}\label{eq:grad_noise_stoch}
        \begin{aligned}
          M^{x^-}_k = G^{x^-}_k - \nabla f(x_k), \quad
          M^x_k = G^x_{k} - \nabla f(x_{k-1}), \quad
          M^y_k =  G^y_k - \nabla f(y_k) .
        \end{aligned}
    \end{equation}
    
    The filtration $\Fscr:=(\Fc_k)_{k\in\NN}$ is a sequence increasing sub-$\sigma$-algebras, where $\Fc_k = \sigma(x_i, \omega_{i-1} \mid i\le k)$. Clearly, $\Fc_k$ is the algorithm history that includes all past iterates $x_0, x_1,\dots,x_k$ and all sampling events up to iteration $k-1$.
    
    One can see that
    \[
    \EE[M^x_ k \mid \Fc_k] =  0 \qquad \text{and} \qquad
    \EE[M^{x^-}_k \mid \Fc_k] =  0.
    \]
    However, this is not the case for $M^y_k$. Indeed, $y_k$ depends on the event $\omega^{x^-,x}_k$, and thus is a random variable conditioned on $\Fc_k$, and so is $\nabla f(y_k)$. Since $M^y_k$ then ultimately depends on $M^x_k$ and $M^{x^-}_k$, conditionally on $\Fc_k$, we have
    \begin{equation}\label{eq:exp_G2}
    \begin{aligned}
    \EE[G^y_ k \mid \Fc_k] =  \frac{1}{N^y_k} \sum_{i=1}^{N^y_k} \EE_{\omega^{x^-,x}_k}[\EE_{\zeta}[\nabla F(y_k,\zeta) \mid \Fc_k,\omega^{x^-,x}_k] \mid \Fc_k] = \EE[\nabla f(y_k) \mid \Fc_k] .
    \end{aligned}
    \end{equation}
    In turn 
    \[
    \EE[M^y_ k \mid \Fc_k] = \EE[G^y_ k \mid \Fc_k] - \EE[\nabla f(y_k) \mid \Fc_k] = 0 .
    \]

\subsection{Stochastic IGAHD}
Generically, the stochastic inertial gradient algorithm we study, coined (S-IGAHD), is presented as Algorithm~\ref{alg:hdds}.

\begin{algorithm}[H]
\caption{Stochastic Inertial Gradient Algorithm with Hessian Driven Damping (S-IGAHD)}\label{alg:hdds}
\textbf{Input:} Initial point $x_0\in\Hc$, step-sizes $s_k \in ]0,1/L]$, parameters $(\alpha,\beta_k)$ satisfying $\alpha\ge 3$, $0<\beta_k<\frac{\sqrt{s_k}}{2}$, and sizes of the minibatches  $(N^{x^-}_k)_{k\in\NN}$, $(N^x_k)_{k\in\NN}$ and $(N^y_k)_{k\in\NN}$;
\begin{algorithmic}[1]
\For{$k=1,2,\ldots$}
\State Set $\alpha_k=1-\frac{\alpha}{k}$. 
\State Draw independent samples $(\zeta^x_i)_{i \in [N^x_k]} \sim_{iid} \mu$, $(\zeta^{x^-}_i)_{i \in [N^{x^-}_k]} \sim_{iid} \mu$.
\State Compute gradient estimates
\begin{equation}\label{eq:batches}
    G^x_ k = \frac{1}{N^x_k} \sum_{i=1}^{N^x_k} \nabla F(x_k,\zeta^x_i) \qquad G^{x^-}_k = \frac{1}{N^{x^-}_k} \sum_{i=1}^{N^{x^-}_k} \nabla F(x_{k-1},\zeta^{x^-}_i)
\end{equation}
\State Update,
\begin{equation}\label{eq:algstepy}
y_k = x_k+\alpha_k(x_k-x_{k-1})-\beta_k \sqrt{s_k}G^x_k + \beta_k\sqrt{s_{k-1}}\left( 1 - \frac{1}{k}\right)
G^{x^-}_k
\end{equation}
\State Draw another independent sample $(\zeta^y_k)_{i \in [N^y_k]} \sim_{iid} \mu$ and compute
\begin{equation}\label{eq:batch_2}
    G^y_ k = \frac{1}{N^y_k} \sum_{i=1}^{N^y_k} \nabla F(y_k,\zeta^y_i) .
\end{equation}
\State Update,
\begin{equation}\label{eq:algstepx}
x_{k+1}=y_k- s_k G^y_k .
\end{equation}
\EndFor
\end{algorithmic}
\end{algorithm}




We will use the following shorthand notation for the quantities:
   \begin{equation}\label{eq:vars} 
   \begin{aligned}
   & \sigma_{x^-,k} = \EE[\|M^{x^-}_k\|^2 \mid \Fc_k]^{1/2}, \qquad \sigma_{x,k} = \EE[\|M^x_k\|^2 \mid \Fc_k]^{1/2},\qquad \sigma_{y,k} = \EE[\|M^y_k\|^2 \mid \Fc_k]^{1/2} .
        \end{aligned}  
        \end{equation}
    We insist on the fact that while $\sigma_{x^-,k}^2$ and $\sigma_{x,k}^2$ are conditional variances of $G^{x^-}_k$ and $G^x_k$ respectively, it is not the case for $\sigma_{y,k}^2$ as it contains a bias term as well. 
    
    Moreover, by independent sampling we have the following 
    \begin{align}
    & \sigma_{x,k}^2 = \frac{1}{N^x_k} \left(\EE_{\zeta}[\|\nabla F(x_k,\zeta)\|^2 \mid \Fc_k] - \norm{\nabla f(x_{k})}^2\right) \leq \frac{1}{N^x_k} \EE_\zeta[\|\nabla F(x_k,\zeta)\|^2 \mid \Fc_k] \quad \text{and} \quad \label{eq:sig1uboundNk}\\
    & \sigma_{x^-,k}^2 \leq \frac{1}{N^{x^-}_k} \EE_\zeta[\|\nabla F(x_{k-1},\zeta)\|^2 \mid \Fc_k] . \label{eq:sig1-uboundNk}
    \end{align}
    On the other hand,
     \begin{align}
    \sigma_{y,k}^2 
    &= \EE[\norm{G^y_k - \nabla f(y_k)}^2 \mid \Fc_k] \nonumber\\
    &= \EE_{\omega^{x^-,x}_k}\left[\EE\left[\norm{G^y_k - \EE\left[G^y_k \mid \Fc_k,\omega^{x^-,x}_k\right]}^2 \mid \Fc_k,\omega^{x^-,x}_k\right] \mid \Fc_k\right] \nonumber\\
    &= \EE_{\omega^{x^-,x}_k}\left[\mathrm{Var}\left[G^y_k \mid \Fc_k,\omega^{x^-,x}_k\right] \mid \Fc_k \right] \nonumber\\
    &= \frac{1}{N^y_k} \EE_{\omega^{x^-,x}_k}\left[\mathrm{Var}_\zeta\left[\nabla F(y_k,\zeta) \mid \Fc_k,\omega^{x^-,x}_k\right] \mid \Fc_k\right] \nonumber\\
    &\leq \frac{1}{N^y_k} \EE[\norm{\nabla F(y_k,\zeta)}^2 \mid \Fc_k] . \label{eq:sig2uboundNk}
    \end{align}

    We will make use of the following spaces of random sequences.
    \begin{definition}
    We denote by \( \ell^+(\Fscr) \) the set of sequences of \( [0, +\infty[ \)-valued random variables \( (a_k)_{k \in \mathbb{N}} \) such that, for each \( k \in \mathbb{N} \), \( a_k \) is \( \mathcal{F}_k \)-measurable. For \( p \in\, ]0, +\infty[ \), we define the set of \( p \)-summable random variables as:
\[
\ell^p_+(\Fscr) := \left\{ (a_k)_{k \in \mathbb{N}} \in \ell^+(\Fscr) \;\middle|\; \sum_{k \in \mathbb{N}} a_k^p < +\infty \quad \Pas \right\}.
\]

The set of non-negative \( p \)-summable deterministic sequences is denoted by \( \ell^p_+ \).
\end{definition}

\subsection{Convergence guarantees}
We start by making a slight change in notation to our Lyapunov function
\begin{equation}\label{eq:lyapunov_stoch}
V_{k} := {s_k}t_k^2 (f(x_k)-\min_{\Hc} 
 f)+\frac{1}{2} \dist(z_k, S)^2, \quad
\text{where} \quad z_k = x_{k-1} + t_{k}\left(x_k - x_{k-1} + \beta_{k-1}\sqrt{s_{k-1}}\nabla f(x_{k-1})\right) ,
\end{equation}
where $\dist(z_k, S)$ is the distance of point $z_k$ from the set of minimizers $S$. 
Moreover, we let
\begin{equation} \label{eq:M_k_stoch}\tag{$K_1$}
\begin{aligned}
M_k &:= \left(\beta_k\sqrt{s_{k}}t_{k+1}M^x_k - \beta_{k-1}\sqrt{s_{k-1}}t_k M^{x^-}_{k}\right) + s_kt_{k+1}M^y_k .
\end{aligned}
\end{equation}

Also, for convenience, we evaluate the following expression to be used later
\begin{align*}
    \EE[\|M_k\|^2 \mid \Fc_k] &= \beta_k^2s_kt_{k+1}^2\EE[\|M^x_k\|^2 \mid \Fc_k] + \beta_{k-1}^2s_{k-1}t_{k}^2\EE[\|M^{x^-}_k\|^2 \mid \Fc_k] + s_k^2t_{k+1}^2\EE[\|M^y_k\|^2 \mid \Fc_k]\\ 
    &\quad +2\beta_ks_k^{3/2}t_{k+1}^2\EE[\langle M^x_k, M^{y}_k\rangle \mid \Fc_k] -2\beta_{k}\beta_{k-1}\sqrt{s_k}\sqrt{s_{k-1}}t_kt_{k+1}\EE[\langle M^{x}_k, M^{x^-}_k\rangle \mid \Fc_k] \\
    &\quad -2\beta_{k-1}s_k\sqrt{s_{k-1}}t_kt_{k+1}\EE[\langle M^{x-}_k, M^{y}_k\rangle \mid \Fc_k] .
\end{align*}
The inner product $\EE[\langle M^{x}_k, M^{x^-}_k\rangle \mid \Fc_k]=0$ by independence. Consider the term $\EE[\langle M^x_k, M^y_k\rangle \mid \Fc_k]$. We have
\begin{equation}\label{eq:noise_cross}
\begin{aligned}
    \EE[\langle M^x_k, M^y_k\rangle \mid \Fc_k] 
    &= \EE[\langle G^y_k - \nabla f(y_k) , M^x_k \rangle \mid \Fc_k]\\
    &= \EE_{\omega^{x^-,x}_k}[\EE[\langle G^y_k - \nabla f(y_k) , M^x_k \rangle \mid \Fc_k,\omega^{x^-,x}_k] \mid \Fc_k]\\
    &= \EE_{\omega^{x^-,x}_k}[\langle \EE_\zeta[G^y_k - \nabla f(y_k) \mid \Fc_k,\omega^{x^-,x}_k], M^x_k \rangle \mid \Fc_k] = 0 .
\end{aligned}
\end{equation}
Similarly, we have $\EE[\langle M^{x^-}_k, M^y_k\rangle \mid \Fc_k] = 0$. Combining this, we rewrite
\begin{equation}\label{eq:norm_Mk}
\EE[\|M_k\|^2 \mid \Fc_k] = \beta_k^2s_k^2t_{k+1}^2 \EE[\|M^x_k\|^2 \mid \Fc_k] + \beta_{k-1}^2s_{k-1}^2t_{k}^2 \EE[\|M^{x^-}_k\|^2 \mid \Fc_k] +  s_k^2t_{k+1}^2\EE[\|M^y_k\|^2 \mid \Fc_k]  .
\end{equation}
We are now ready to establish the convergence properties of Algorithm~\ref{alg:hdds}.
\begin{theorem}\label{thm:conv_sto}
Let $f$ be as given \eqref{eq:poprisk} under the required assumptions on $F$. Assume that $S := \argmin_{\Hc} f\neq \emptyset$. Let $(x_k)_{k\in\NN}$ be the sequence generated by Algorithm~\ref{alg:hdds} with $\alpha \geq 3$, where \(s_k \in ]0, 1/L]\) is a nonincreasing sequence and $0 < \beta_k < \frac{\sqrt{s_k}}{2}$, and the stochastic gradient estimates are such that:
\begin{equation} \tag{$K_2$} \label{eq:k2}
\begin{aligned}
    (s_kk\sigma_{x,k})_{k \in \NN} &\in \ell^2_{+}(\Fscr), \qquad 
    (s_kk\sigma_{x^-,k})_{k \in \NN} &\in \ell^2_{+}(\Fscr), \qquad 
    (s_kk\sigma_{y,k})_{k \in \NN} &\in \ell^2_{+}(\Fscr).
\end{aligned}
\end{equation}
\begin{enumerate}[label=(\roman*)]
    \item We have the following convergence rates on the objective values: \label{item:vals_rate}
    \begin{equation}\label{eq:as_vals_rate}
        f(x_k) - \min_{\mathcal{H}} f = \mathcal{O}\left(\frac{1}{s_k k^2}\right) \text{ as } k \to \infty \quad \Pas ,
    \end{equation}
    and
    \begin{equation}\label{eq:exp_vals_rate}
    \mathbb{E}\left[f(x_k) - \min_{\mathcal{H}} f\right] \leq \frac{\frac{(\alpha-1)^2}{2} \operatorname{dist}(x_0, S)^2 + \sum_{i=1}^{k} i^2 e_i}{s_k (k-1)^2}
    \end{equation}
    where 
    \begin{equation}\label{eq:edefn}
        e_k := 4\beta_k^2s_k\sigma_{x,k}^2+ \beta_{k-1}^2s_{k-1}\sigma_{x^-,k}^2 + 4s_k^2\sigma_{y,k}^2
    \end{equation}
    
    \item \label{item:grads}
    The gradient sequence obeys the summability properties: 
    \begin{equation}\label{eq:exp_grads_sum}
    \hspace*{-0.75cm}
    \sum_{k \in \NN} s_k k^2\, \mathbb{E}[\|\nabla f(y_k)\|^2 \mid \mathcal{F}_k] < +\infty ~ \Pas, ~~ \sum_{k \in \mathbb{N}} s_k k^2 \mathbb{E} [\|\nabla f(y_k)\|^2 ] < +\infty ~~ \text{and} ~~ \sum_{k \in \mathbb{N}} s_k k^2 \mathbb{E} [\|\nabla f(x_k)\|^2 ] < +\infty.
    \end{equation}
    Thus,
    \begin{equation}\label{eq:as_grads_o}
    \|\nabla f(y_k)\|^2=o\left(\frac{1}{s_k k^2}\right)\quad\Pas\quad \text{and}\quad \|\nabla f(x_k)\|^2=o\left(\frac{1}{s_k k^2}\right)\quad \Pas .
    \end{equation}

\item If, moreover, $\alpha > 3$, then
\label{item:sum_vals}
    \begin{align}\label{eq:stochres_f_k}
        & \sum_{k \in \NN} s_k k \left(f\left(x_k\right) - \min_{\Hc} f\right) < +\infty \quad \Pas . 
    \end{align}

    \item \label{item:velsto}
    Suppose also that $\alpha > 3$, $s_k \equiv s \in ]0, 1/L]$, $\beta_k \equiv \beta \in ]0,\frac{\sqrt{s}}{2}[$, and \eqref{eq:k2} is replaced by
    \begin{equation} \tag{$K_2^+$} \label{eq:k2+}
    \begin{aligned}
    \left(k\EE[\norm{M^x_k}^2]^{1/2}\right)_{k \in \NN} &\in \ell^1_{+}, \qquad 
    \left(k\EE[\norm{M^{x-}_k}^2]^{1/2}\right)_{k \in \NN} &\in \ell^1_{+}, \qquad 
    \left(k\EE[\norm{M^y_k}^2]^{1/2}\right)_{k \in \NN} &\in \ell^1_{+} .
\end{aligned}
\end{equation}
    Then
    \begin{enumerate}
    \item $\sup _{k \in \NN} k\EE[\left\|x_{k}-x_{k-1}\right\|]<+\infty$.\label{item:vel_bnd_sto}
    \item $\sum_{k \in \NN} k\EE[\left\|x_{k}-x_{k-1}\right\|^{2}]<+\infty$ and thus $k\norm{x_{k}-x_{k-1}}^2 \to 0 ~ \Pas$\label{item:vel_sum_sto}.
    \item $\EE[f(x_k) - \min_{\Hc} f] = o\left(\frac{1}{k^2}\right)$ and $\EE[\|x_k - x_{k-1}\|] = o\left(\frac{1}{k}\right)$ as $k \rightarrow +\infty$.  \label{item:valvel_fast_o_sto}
    \end{enumerate}
    



\end{enumerate}
\end{theorem}

\noindent Before proving the result, we present a few remarks interpreting the main insight and conclusions of the theorem, which can be informative for choosing the appropriate algorithm parameters to the setting and convergence desiderata of interest. 


\begin{remark}\label{rem:conv_sto_1}
Unlike the deterministic case,  we have modified the Lyapunov function to make it independent of the choice of the minimizer $x^\star \in S$. The reason behind this is to make the almost sure claims independent of that choice, as otherwise, the set of events of probability one on which these claims hold would depend on $x^\star$, which may not be unique in the general case. This is an important point to keep in mind in the stochastic setting.
\end{remark}

\begin{remark}[Rates and variance summability]\label{rem:conv_sto_2}
When the gradients are bounded \Pas (e.g., logistic loss, or if the sequences are a priori bounded \Pas), then for \eqref{eq:k2} to hold, it follows from \eqref{eq:sig1uboundNk}, \eqref{eq:sig1-uboundNk} and \eqref{eq:sig2uboundNk} that it is sufficient to assume that $\sum_{k \in \NN} \frac{s_k^2k^2}{N_k} < +\infty$, where $N_k=\min(N^x_k,N^{x^-}_k,N^y_k)$. One can then choose a sufficiently fast increasing minibatch size and/or a step-size sequence that decreases fast enough. For instance, one needs $N_k$ to grow at least as $k^{3-\epsilon}$ if $s_k=s_0/(k+1)^{\epsilon}$ for $\epsilon \in ]0,1/2]$ and $s_0 > 0$. For $\epsilon$ small but bounded away from zero, this would give a convergence rate $\mathcal{O}(1/k^{2-\epsilon})$ on the objective values in expectation. For a smaller minibatch size growing only as $k^{2-\delta}$, $\delta \in [0,1/2]$, one has to take $s_k=s_0/(k+1)^{(1+\epsilon)/2}$ for $\epsilon \in ]\delta,1/2]$. This comes at the price of the slower convergence rate $\mathcal{O}(1/k^{(3-\epsilon)/2})$.

If the minibatch sizes and/or step-sizes do not increase/decrease fast enough, then \eqref{eq:exp_vals_rate} tells us that the objective values converge in expectation to a noise-dominated region around the optimal value, and the size of this region depends on the second moment of the errors. If the error is constant, i.e., the minibatch size is fixed to say $N$,  $\beta_k$ is set to say $\sqrt{s_k}/8$, and the gradients are bounded so that the variances are $O(1/N)$ (see \eqref{eq:sig1uboundNk}, \eqref{eq:sig1-uboundNk} and \eqref{eq:sig2uboundNk}), then one has
\[ 
\EE\left[f(x_k) - \min_{\mathcal{H}} f\right] \leq \mathcal{O}\left(\frac{1}{s_k (k-1)^2}\right) + \mathcal{O}\left(\frac{\sum_{i=0}^{k-1} (i+1)^2 s_{i}^2}{N s_k (k-1)^2}\right) .
\]
For $s_k=s_0/(k+1)$, one has
\[ 
\EE\left[f(x_k) - \min_{\mathcal{H}} f\right] \leq \mathcal{O}\left(\frac{1}{k} + \frac{1}{N}\right) .
\]
To solve the problem to accuracy $\epsilon > 0$, one can run the algorithm with $N \gtrsim \epsilon^{-1}$ for $k \gtrsim \epsilon^{-1}$ iterations.
\end{remark}

\begin{remark}[Convergence of iterates]
Unlike the deterministic case, we do not have a proof of almost sure weak convergence of iterates $(x_k)_{k \in \NN}$ in the general stochastic setting. The deep reason is that $y_k$ is a random variable even conditioned on the past $\Fc_k$. Consequently, we do not have a sufficient control on the summability of the gradients neither at $y_k$ nor at $x_k$ in the almost sure sense, but only in (conditional) expectation. This summability control in almost sure sense is instrumental in order to invoke Opial's lemma. In the deterministic setting, such a control allowed to prove rates and summability claims on the velocity sequence $(\|x_k - x_{k-1}\|)_{k \in \NN}$ whence we deduce the boundedness of the iterates $(x_k)_{k \in \NN}$ and were able to invoke Opial's lemma. While we have control on the velocity sequence in expectation, transferring these arguments in the almost sure sens is far from straightforward. Consequently, the weak convergence of the iterates of the stochastic Algorithm~\ref{alg:hdds} remains an open question. Nevertheless, under the additional assumption that the iterates are almost surely bounded, one can show almost weak convergence of the iterates. This will be the subject of Section~\ref{subsec:conv_iter}.
\end{remark}

\begin{proof}

Denote
\[
v_k = z_k - x^\star ,
\]
where $x^\star$ is the closest element in $S$ to $z_k$. We first start from \eqref{eq:lyapmainineq} and using the definition of $E_k$ with a slight change of notation using $M^{x^-}_k$ instead of $M^x_{k-1}$ as this noise component in the stochastic setting is not carried over from past iterate but is due to sampling at the current time. We get
\begin{align*}
s_{k+1}t_{k+1}^2(f(x_{k+1})& - \min_{\Hc} f) - s_{k}t_{k}^2(f(x_{k}) - \min_{\Hc} f) + \frac{1}{2}\|v_{k+1}\|^2 - \frac{1}{2}\|v_{k}\|^2+ \langle v_{k+1}, M_k\rangle
\leq \\
&s_k\left(t_{k+1}^2 - t_{k+1} - t_{k}^2 \right)\left( f(x_k) - \min_{\Hc} f\right)- \frac{\epsilon_k}{2}s_kt_{k+1}^2 \left\|\nabla f(y_k)\right\|^2 \\
&- s_kt_{k+1}\left\langle  \left(\beta_k\sqrt{s_{k}}t_{k+1}M^x_k - \beta_{k-1}\sqrt{s_{k-1}}t_k M^{x^-}_{k}\right), \nabla f(y_k) \right\rangle \\
&- s_kt_{k+1}\left\langle  \left(\beta_k\sqrt{s_{k}}t_{k+1}M^x_k - \beta_{k-1}\sqrt{s_{k-1}}t_k M^{x^-}_{k}\right), M^y_k \right\rangle .
\end{align*}

Now, using the fact 
\[
\dist(z_{k+1}, S) \leq \norm{v_{k+1}} \quad \text{and} \quad \dist(z_k, S) = \norm{v_k}
\]
and the definition of $V_k$, we obtain
\begin{align*}
V_{k+1} - V_{k} +\langle v_{k+1}, M_k\rangle
&\leq s_k\left(t_{k+1}^2 - t_{k+1} - t_{k}^2 \right)\left( f(x_k) - \min_{\Hc} f\right)
- \frac{\epsilon_k}{2}s_kt_{k+1}^2 \left\|\nabla f(y_k)\right\|^2 \\
&- s_kt_{k+1}\left\langle  \left(\beta_k\sqrt{s_{k}}t_{k+1}M^x_k - \beta_{k-1}\sqrt{s_{k-1}}t_k M^{x^-}_{k}\right), \nabla f(y_k) \right\rangle \\
&- s_kt_{k+1}\left\langle  \left(\beta_k\sqrt{s_{k}}t_{k+1}M^x_k - \beta_{k-1}\sqrt{s_{k-1}}t_k M^{x^-}_{k}\right), M^y_k \right\rangle .
\end{align*}

Replacing $v_{k+1}$ by $v_k - s_kt_{k+1}\nabla f(y_k) - M_k$ from \eqref{eq:vk+1vk}, and using the definition of $M_k$ on the rhs, we get
\begin{align*}
V_{k+1} - V_{k} +\langle v_k, M_k\rangle -s_kt_{k+1}\langle \nabla f(y_k), M_k\rangle -\norm{M_k}^2
&\leq s_k\left(t_{k+1}^2 - t_{k+1} - t_{k}^2 \right)\left( f(x_k) - \min_{\Hc} f\right)\\
&- \frac{\epsilon_k}{2}s_kt_{k+1}^2 \left\|\nabla f(y_k)\right\|^2 \\
&- s_kt_{k+1}\left\langle  M_k - s_kt_{k+1}M^y_k, \nabla f(y_k) \right\rangle\\
&- s_kt_{k+1}\left\langle  \left(\beta_k\sqrt{s_{k}}t_{k+1}M^x_k - \beta_{k-1}\sqrt{s_{k-1}}t_k M^{x^-}_{k}\right), M^y_k \right\rangle .
\end{align*}

Using the definition of $M_k$ from \eqref{eq:M_k_stoch} and simplifying, we get
\begin{align*}
V_{k+1} - V_{k} 
&\leq s_k\left(t_{k+1}^2 - t_{k+1} - t_{k}^2 \right)\left( f(x_k) - \min_{\Hc} f\right)- \frac{\epsilon_k}{2}s_kt_{k+1}^2 \left\|\nabla f(y_k)\right\|^2 \\
& \quad+s_k^2t_{k+1}^2\left\langle M^y_k, \nabla f(y_k) \right\rangle - \langle \beta_k\sqrt{s_k}t_{k+1} M^x_k, s_kt_{k+1}M^y_k\rangle + \langle \beta_{k-1}\sqrt{s_{k-1}}t_{k} M^{x^-}_k, s_kt_{k+1}M^y_k\rangle\\
&\quad+ \| \beta_k \sqrt{s_{k}} t_{k+1} M^x_k - \beta_{k-1} \sqrt{s_{k-1}}t_{k}M^{x^-}_k + s_k t_{k+1} M^y_k \|^2 -\langle v_k, M_k\rangle
\end{align*}

Taking expectation conditionally on $\Fc_k$ gives
\begin{equation}\label{eq:cond_exp_Vk}
\begin{aligned}
    \EE[V_{k+1}\mid \Fc_k] &\leq V_k + s_k\left(t_{k+1}^2 - t_{k+1} - t_{k}^2 \right)\left( f(x_k) - \min_{\Hc} f\right) -\frac{\epsilon}{2}s_kt_{k+1}^2\EE[\|\nabla f(y_k)\|^2 \mid \Fc_k] \\
    &\quad + s_kt_{k+1}^2\EE[\langle M^y_k, \nabla f(y_k) \rangle \mid \Fc_k] - \beta_k s_k^{3/2}t_{k+1}^2 \EE[\langle M^x_k, M^y_k\rangle \mid \Fc_k]\\
    &\quad+ \beta_k\beta_{k-1} s_k\sqrt{s_{k-1}}t_{k+1}t_k \EE[\langle M^{x^-}_k, M^y_k\rangle \mid \Fc_k]\\
    &\quad+\EE[\| M_k \|^2 \mid \Fc_k] - \EE[\langle v_k, M_k \rangle \mid \Fc_k] .
\end{aligned}
\end{equation}

Using \eqref{eq:noise_cross}, we can write $\EE[\langle M^x_k, M^y_k\rangle \mid \Fc_k]=0$ and $\EE[\langle M^{x^-}_k, M^y_k\rangle \mid \Fc_k]=0$. Moreover, using the expression for $\EE[\|M_k\|^2 \mid \Fc_k]$ from \eqref{eq:norm_Mk}, we can write
%
\begin{align*}
    \EE[V_{k+1}\mid \Fc_k] &\leq V_k + s_k\left(t_{k+1}^2 - t_{k+1} - t_{k}^2 \right)\left( f(x_k) - \min_{\Hc} f\right) -\frac{\epsilon}{2}s_kt_{k+1}^2\EE[\|\nabla f(y_k)\|^2 \mid \Fc_k] \\
    &\quad + s_kt_{k+1}^2\EE[\langle M^y_k, \nabla f(y_k) \rangle \mid \Fc_k] + \beta_k^2s_k^2t_{k+1}^2 \EE[\|M^x_k\|^2 \mid \Fc_k] + \beta_{k-1}^2s_{k-1}^2t_{k}^2 \EE[\|M^{x^-}_k\|^2 \mid \Fc_k]\\
    &\quad+ s_k^2t_{k+1}^2 \EE[\|M^y_k\|^2 \mid \Fc_k] .
\end{align*}
In the above inequality, we expanded the term $\|M_k\|^2$ and used the fact that $\EE[\langle M^x_k, M^{x^-}_k \rangle \mid \Fc_k]=0$, and that $\EE[\langle v_k, M_k \rangle \mid \Fc_k]=0$ since, conditionally on $\Fc_k$, $v_k$ is deterministic and $M_k$ is zero-mean.

We then have that
\begin{align*}
    \EE[\left\langle  M^y_k, \nabla f(y_k) \right\rangle \mid \Fc_k] 
    &= \EE[\left\langle  G^y_k, \nabla f(y_k) \right\rangle \mid \Fc_k] - \EE[\norm{\nabla f(y_k)}^2 \mid \Fc_k] \\
    &= \EE_{\omega^{x^-,x}_k}\left[\EE\left[\left\langle  G^y_k, \nabla f(y_k) \right\rangle \mid \Fc_k,\omega^{x^-,x}_k\right] \mid \Fc_k\right] - \EE[\norm{\nabla f(y_k)}^2 \mid \Fc_k] \\
    &= \EE_{\omega^{x^-,x}_k}\left[\left\langle \EE_\zeta\left[G^y_k \mid \Fc_k,\omega^{x^-,x}_k\right], \nabla f(y_k) \right\rangle \mid \Fc_k\right] - \EE[\norm{\nabla f(y_k)}^2 \mid \Fc_k] \\
    &= \EE[\norm{\nabla f(y_k)}^2 \mid \Fc_k] - \EE[\norm{\nabla f(y_k)}^2 \mid \Fc_k] = 0 .
\end{align*}




Observing that the assumption~\eqref{eq:k1} implies $t_{k+1}\leq2t_k$, and using \eqref{eq:k2}, and applying the definition of $e_k$ ~\eqref{eq:edefn}, we arrive at
\begin{equation}\label{eq:stoch_iter}
\EE[V_{k+1} \mid \Fc_k] \leq V_{k}
+ s_k\left(t_{k+1}^2 - t_{k+1} - t_{k}^2 \right)\left( f(x_k) - \min_{\Hc} f\right)- \frac{\epsilon_k}{2}s_kt_{k+1}^2 \EE[\left\|\nabla f(y_k)\right\|^2 \mid \Fc_k] + t_{k+1}^2e_k .
\end{equation}
Using \eqref{eq:k1}, the second term of rhs in \eqref{eq:stoch_iter} is nonpositive and can then be dropped which entails
\begin{align*}
\EE[V_{k+1}\mid \Fc_k] &\leq V_{k}
- \frac{\epsilon_k}{2}s_kt_{k+1}^2 \EE[\left\|\nabla f(y_k)\right\|^2 \mid \Fc_k] + t_{k+1}^2e_k .
\end{align*}
Now thanks to the second part of \eqref{eq:k2} and the fact that $V_k$ is bounded from below, we are in position to apply Lemma~\ref{lemma:A3} to say that $V_k$ converges \Pas, and in particular it is bounded, i.e., there exists a $[0, \infty[$-valued random variable $\nu$ such that $\sup_{k\in \NN}V_k \le \nu<+\infty$ \Pas. Therefore, for all $k \in \NN$ 
\begin{equation}\label{eq:value_rates_stoch}
\begin{aligned}
s_k^2t_k^2\left(f(x_k) - \min f\right) \leq V_k \leq \nu < +\infty \quad \Pas ,
\end{aligned}
\end{equation}
which shows \eqref{eq:as_vals_rate}. Lemma~\ref{lemma:A3} also subsequently shows that 
\begin{equation}\label{eq:grad_rates_stoch}
\begin{aligned} \sum\limits_{k \in \NN}s_kt_{k+1}^2\EE[\|\nabla f(y_k)\|^2 \mid \Fc_k]  < +\infty  \quad \Pas .
\end{aligned}
\end{equation}
which is the first part of \eqref{eq:exp_grads_sum}. Now going back to \eqref{eq:stoch_iter}, taking the total expectation and summing, we get
\begin{equation}\label{eq:expec_grad_rates_stoch}
\begin{aligned} 
\sum\limits_{k \in \NN}s_kt_k^2\EE[\|\nabla f(y_k)\|^2] \leq \sum\limits_{k \in \NN}s_kt_{k+1}^2\EE[\|\nabla f(y_k)\|^2]  < +\infty ,
\end{aligned}
\end{equation}
and
\begin{align*}
s_{k}^2t_{k}^{2} \EE\left[f\left(x_{k}\right)-\min_{\Hc} f\right] 
\leq \EE\left[V_{k}\right] \leq V_{0} + \sum \limits_{i=0}^kt_{i+1}^{2} e_{i} 
\leq s_0t_{0}^{2} (f\left(x_{0}\right)-\min_{\Hc} f) + \frac{1}{2} \dist(x_0,S)^2 + \sum\limits_{i=0}^{k} t_{i+1}^{2} e_{i} ,
\end{align*}
where we used the fact that $s_k$ is non-increasing and $x_{0} = x_{-1}$, $s_{-1}=0$ and $\alpha_0=0$. 
This shows \eqref{eq:exp_vals_rate} and the second part of \eqref{eq:exp_grads_sum}. For the third summability claim of  \eqref{eq:exp_grads_sum}, we start using Lipschitz continuity of $\nabla f$ and \eqref{eq:algstepx} to see that
\[
\norm{\nabla f(x_{k+1})}^2 \leq 2(1+Ls_k)^2\norm{\nabla f(y_k)}^2 + 2L^2s_k^2\norm{M^y_k}^2 .
\] 
Taking expectation on both sides and multiplying by $s_kt_{k+1}^2$, we get
\begin{equation}\label{eq:exp_as}
\begin{aligned}
s_kt_{k+1}^2\EE[\norm{\nabla f(x_{k+1})}^2] \leq 2(1+Ls_k)^2s_kt_{k+1}^2\EE[\norm{\nabla f(y_k)}^2] + 2L^2s_0s_k^2t_{k+1}^2\EE\left[\|M^y_k\|^2\right] .
\end{aligned}
\end{equation}
The first term on the right hand side is summable (see \eqref{eq:expec_grad_rates_stoch}). The second one is also summable using the tower property of conditional expectation, and Fubini-Tonelli's Theorem together with \eqref{eq:k2}. Since $s_k$ is nonincrasing, we get the claim.

%
%
In turn, \eqref{eq:as_grads_o} follows from \eqref{eq:exp_grads_sum} and Lemma~\ref{lemma:meanas}. 

To prove \eqref{eq:stochres_f_k}, we start from \eqref{eq:stoch_iter} and use that when $\alpha > 3$, then $t_{k+1}^2 - t_k^2 \leq m t_{k+1}$ with $m \in ]0,1[$. Thus dropping the negative term $- \frac{\epsilon_k}{2}s_kt_{k+1}^2 \EE[\left\|\nabla f(y_k)\right\|^2 \mid \Fc_k]$,  we have
\begin{equation}\label{eq:stoch_iter_2}
\EE\left[V_{k+1} \mid \Fc_k\right] \leq V_{k} - s_k(1-m) t_{k+1}\left( f(x_k) - \min_{\Hc} f\right) + t_{k+1}^2e_k .
\end{equation}
We can again invoke Lemma~\ref{lemma:A3} to get the desired claim.

We now turn to proving the claims of \ref{item:velsto}. We embark from \eqref{eq:dkineq} that we write in the stochastic setting as
\begin{align*}
\frac{K^2}{2} d_{K+1} 
&\leq (k_0 - 1)^2 \theta_{k_0} + (k_0 - 1)^2 d_{k_0} + 2 \sum\limits_{k=k_0}^K k \theta_k +  \beta^2 \sum\limits_{k=k_0}^K (k+1)^2 g_k^2 
+ \sum\limits_{k=k_0}^K m_k k\|x_k - x_{k-1}\| 
+ s^4K^2\norm{M^y_k}^2 , 
\end{align*}
where $m_k := \sqrt{s}\beta\alpha_k k \left(\|{M^x_k}\| + \|{M^{x^-}_k}\|\right) + s^2k\|{M^y_{k-1}}\|$, $g_k := \sqrt{s}\left(\|G^x_k\| + \|G^{x^-}_k\|\right)$. Taking the full expectation on both sides and using H\"older's inequality, we have
\begin{align*}
\frac{K^2}{2} \EE[d_{K+1}]
&\leq (k_0 - 1)^2 \theta_{k_0} + (k_0 - 1)^2 d_{k_0} + 2 \sum\limits_{k=k_0}^K k \EE[\theta_k] +  \beta^2 \sum\limits_{k=k_0}^K (k+1)^2 \EE[g_k^2] \nonumber\\
&+ \sum\limits_{k=k_0}^K \EE[m_k^2]^{1/2} k\EE[\|x_k - x_{k-1}\|^2]^{1/2} 
+ s^4K^2\EE[\norm{M^y_k}^2] . 
\end{align*}
Taking the full expectation in \eqref{eq:stoch_iter_2} we have
\begin{align}\label{eq:stoch_iter_exp}
\EE[V_{k+1}] \leq \EE[V_{k}] - s(1-m) t_{k+1}\EE[\theta_k] + t_{k+1}^2 (4\beta^2s\EE[\|M^x_k\|^2] + \beta^2s\EE[\|M^{x^-}_k\|^2] + 4s^2\EE[\|M^y_k\|^2]) .
\end{align}
We have
\[
\sum_{k=0}^Kt_{k+1}^2\EE[\|M^x_k\|^2] \leq \left(\sum_{k=0}^Kt_{k+1}\EE[\|M^x_k\|^2]^{1/2}\right)^{2},
\]
and thus the last term in the rhs of \eqref{eq:stoch_iter_exp} is summable thanks to \eqref{eq:k2+}. It then follows that $(k\EE[\theta_k])_{k \in \NN} \in \ell^1_+$. We also have
\[
g_k^2 \leq 4\left(\norm{\nabla f(x_k)}^2 + \norm{\nabla f(x_{k-1})}^2 + \norm{\nabla f(x_k)}^2 + \|M^x_k\|^2 + \|M^{x^-}_k\|^2\right) .
\]
Arguing as above for the error terms and using \eqref{eq:exp_grads_sum}, we infer that $(k^2\EE[g_k^2])_{k \in \NN} \in \ell^1_+$. Therefore, there exists a constant $C_1 > 0$ such that
\begin{align*}
(K\EE[\|x_{K+1} - x_K\|^2]^{1/2})^2
&\leq 4C_1 + 4\sum\limits_{k=k_0}^K \EE[m_k^2]^{1/2} k\EE[\|x_k - x_{k-1}\|^2]^{1/2} . 
\end{align*}
We have
\[
\EE[m_k^2]^{1/2} \leq C_2\left( k\EE[\|M^x_k\|^2]^{1/2} + k\EE[\|M^{x^-}_k\|^2]^{1/2} + k\EE[\|M^y_k\|^2]^{1/2}\right) ,
\]
for a constant $C_2 > 0$, and thus $(\EE[m_k^2]^{1/2})_{k \in \NN} \in \ell^1_+$ thanks to \eqref{eq:k2+}. We are then in position to invoke Lemma~\ref{lemma:A2} and Jensen's inequality to get \ref{item:vel_bnd_sto}. Let $C_3 < +\infty$ the uniform bound in \ref{item:vel_bnd_sto}.

We now argue similarly to the proof of Theorem~\ref{thm:vals_thm}\ref{item:vel_second}. Let $W_{k} :=(k-1)^{2} \EE[d_k]+(k-1)^{2} \EE[\theta_{k}]$. Starting from \eqref{eq:recursive_vel}, taking the full expectation, we have
\begin{align*}
W_{k+1} + \left(2(\alpha-1)k - \alpha^2\right)\EE[d_k] 
\leq W_k + q_k,
\end{align*}
where $q_k = 2k\EE[\theta_k] + \beta^2(k+1)^2\EE[g_k]^2 + C_3\EE[m_k^2]^{1/2} + C_3s^2k\EE[\norm{M^y_k}^2]^{1/2}$. For $k \geq k_0$ large enough, the second term on the lhs is nonnegative. We have already shown above that $\sum_{k \in \NN} q_k <+\infty$ thanks to \eqref{eq:k2+}. Applying Lemma~\ref{lemma:A3-det} shows that \ref{item:vel_sum_sto} holds and $W_{k}$ converges as $k \rightarrow+\infty$. We have already shown that $(k\EE[\theta_k])_{k \in \NN} \in \ell^1_+$. This together with claim~\ref{item:vel_sum_sto} ensure that $(W_k/k)_{k \in \NN}$ is summable. Upon invoking Lemma~\ref{lemma:Asumo}, we get that $W_k \to 0$ when \ref{item:valvel_fast_o_sto} follows immediately.
\end{proof}

\subsection{Convergence of the iterates}\label{subsec:conv_iter}
To prove convergence of iterates, we will rely on the law of the iterated logarithm (LIL) for iid variables. Recall the Lipschitz continuity modulus function $l: \Xi \to \RR^+$ from \eqref{eq:assumpoprisk}. We will also make the following assumption
\begin{equation}\label{eq:assumlil}
\EE[\norm{\nabla F(0,\zeta)}^2] < \infty .
\end{equation}
Observe that by standard embedding of $L^p(\Hc,\mu)$ spaces, the first part of this assumption implies that $l$ is also integrable wrt $\mu$ as we required in the beginning.

We start by showing the following lemma.
\begin{lemma}\label{lem:lil}
Assume \eqref{eq:assumpoprisk} and \eqref{eq:assumlil} hold. Let $\zeta_1,\zeta_2,\ldots$ be iid from $\mu$. For any fixed $x \in \RR^n$, the following holds
\[
\limsup_{N \to +\infty} \frac{\norm{\sum_{i=1}^{N} (\nabla F(x,\zeta_i) - \nabla f(x))}}{\sqrt{2N\log(\log(N))}} \leq \EE[\norm{\nabla F(0,\zeta)}^2]^{1/2} + \norm{x}\EE[l(\zeta)^2]^{1/2} < +\infty \quad \text{almost surely} .
\]
\end{lemma}

\begin{proof}
Let $v$ be a vector on the unit sphere of $\Hc$. Define $Z_i := \dprod{v,\nabla F(x,\zeta_i)}$ and $z := \dprod{v,\nabla f(x)}$. It follows by assumption on the $\zeta_i$'s that the $Z_i$'s are also iid. Moreover, $\EE[Z_i] = z$ and the variance is
\[
\sigma_Z^2 := \mathrm{Var}[Z_1] = \EE[\dprod{v,\nabla F(x,\zeta_1) - \nabla f(x)}^2] \leq \EE[\norm{\nabla F(x,\zeta_1)}^2] . 
\]
Invoking Minkowski's inequality and $l(\zeta)$-Lipschitz continuity of $\nabla F(\cdot,\zeta)$, we get
\[
\sigma_Z \leq \EE[\norm{\nabla F(0,\zeta_1)}^2]^{1/2} + \norm{x} \EE[l(\zeta_1)^2]^{1/2}. 
\]
We now apply the LIL (in its Khintchine form) to the random variables $Z_i$, see e.g., \cite{Bingham86}, to get that 
\[
\limsup_{N \to +\infty} \frac{\sum_{i=1}^{N} \dprod{v,\nabla F(x,\zeta_i) - \nabla f(x)}}{\sqrt{2N\log(\log(N))}} = \sigma_Z \quad \text{almost surely} .
\]
Taking the supremum over $v$ on the unit sphere and the upper bound on $\sigma_Z$ above, we get the claim.
\end{proof}

Our assumption~\eqref{eq:assumlil} is mild and holds many situations of interest. These include the two examples considered in Section~\ref{s:num} as we now show (we recall the notations from there). 

\begin{example}[Logistic loss]
In this case, $f$ has the form \eqref{eq:logistic_pop} and thus, $\zeta = (\yb,\phib)$, $\Hc = \RR^n$ and
\[
F(x,\zeta) = \yb \log(1+e^{-\dprod{x,\phib}}) + (1 - \yb) \log(1+e^{\dprod{x,\phib}}) .
\]
It is immediate to see that
\[
\nabla F(x,\zeta) = \frac{-\yb+(1-\yb) e^{\dprod{x,\phib}}}{1+e^{\dprod{x,\phib}}}\phib .
\]
Since $\yb$ takes values in $\{0,1\}$, we get that $\forall x \in \RR^n$,
\[
\EE[\norm{\nabla F(x,\zeta)}^2] \leq \EE[\norm{\phib}^2] .
\]
Clearly, from the proof of Lemma~\ref{lem:lil}, we actually have in this case that $\sigma_Z^2 \leq \EE[\norm{\phib}^2]$. Thus, if $\EE[\norm{\phib}^2] < +\infty$, then the conclusion of the lemma can be improved to
\[
\limsup_{N \to +\infty} \frac{\norm{\sum_{i=1}^{N} (\nabla F(x,\zeta_i) - \nabla f(x))}}{\sqrt{2N\log(\log(N))}} \leq \EE[\norm{\phib}^2]^{1/2} < +\infty \quad \text{almost surely} ,
\]
independently of $x$ and $l$. 

\end{example}

\begin{example}[Quadratic loss]
Here we consider $\Hc=\RR^n$ and $\zeta = (\yb,\phib)$ with
\[
F(x,\zeta) = \frac{1}{2}\norm{\yb - \dprod{x,\phib}}^2 ,
\]
where $\yb$ takes values in $\RR$. It is easy to see that
\[
\nabla F(x,\zeta) = \phib\phib^\top x - \phib \yb .
\]
Clearly $\nabla F(0,\zeta) = -\phib \yb$ and $l(\zeta) = \norm{\phib}^2$. Therefore, for assumption~\eqref{eq:assumlil} to hold, it is sufficient that
\[
\EE[\norm{\phib}^4] < +\infty \quad \text{and} \quad \EE[\norm{\yb}^2] < +\infty .
\]
\end{example}

\begin{theorem}\label{thm:conv_it}
Assume that $\Hc$ is a separable real Hilbert space. Let $f$ be as given \eqref{eq:poprisk} under the required assumptions on $F$ and $l$ and such that \eqref{eq:assumlil} also holds. Assume that $S := \argmin_{\Hc} f\neq \emptyset$. Let $(x_k)_{k\in\NN}$ be the sequence generated by Algorithm~\ref{alg:hdds} with $\alpha \geq 3$, $\beta_k=\eta\frac{\sqrt{s_{k}}}{2}$, $\eta \in ]0,1]$, and $s_{k} \in ]0,1/L]$ is nonincreasing. Let $N_k=\min(N^x_k,N^{x-}_k,N^y_k)$ and assume that
\[
\sum_{k \in \NN} s_k k\sqrt{\frac{\log(\log(N_k))}{N_k}} < +\infty .
\]
Suppose that $(x_k)_{k\in\NN}$ is almost surely bounded. Then the following holds:
\begin{enumerate}[label=(\roman*)]
    \item $f\left(x_{k}\right)-\min_{\Hc} f=\mathcal{O}\left(\frac{1}{s_k k^{2}}\right) \text{ as } k \rightarrow \infty$ \Pas.\label{item:val_first_sto_det}
    
    \item We have 
    \[
    \sum_{k \in \NN} s_k k^2 \|\nabla f(y_k)\|^2 < +\infty \quad \Pas.
    \]
    If $\inf_{k \in \NN} s_k > 0$, then
    \[
    \sum_{k \in \NN} k^2 \|\nabla f(x_k)\|^2 < +\infty \quad \text{and} \quad \sum_{k \in \NN} k^2 \|\nabla f(y_k)\|^2 < +\infty \quad \Pas.
    \]
    \label{item:val_second_sto_det}

    Suppose now that $\alpha > 3$. Then
    \item $\sum_{k \in \NN} s_k k(f\left(x_{k}\right)-\min_{\Hc} f) < +\infty$ \Pas.\label{item:sto_sumf}
    \item If $s_k \equiv s > 0$, then
    \label{item:sto_vel}
    \begin{enumerate}
    \item $\sup_{k \in \NN} k\left\|x_{k}-x_{k-1}\right\|<+\infty$ \Pas. \label{item:vel_sto_first}
    \item $\sum_{k \in \NN} k\left\|x_{k}-x_{k-1}\right\|^{2}<+\infty$ \Pas. \label{item:vel_sto_second}
    \item $\left(x_{k}\right)_{k \in \NN}$ converges weakly \Pas ~to an $S$-valued random variable.\label{item:sto_iter}
    \end{enumerate}
\end{enumerate}
\end{theorem}

\begin{proof}
Applying Lemma~\ref{lem:lil} to $M^{x^-}_k$ in \eqref{eq:grad_noise_stoch}, we have, almost surely, that
\[
\limsup_{N^{x^-}_k \to +\infty} \sqrt{\frac{N^{x^-}_k}{2\log(\log(N^{x^-}_k))}}\norm{M^{x^-}_k} \leq \EE[\norm{\nabla F(0,\zeta)}^2]^{1/2} + \nu\EE[l(\zeta)^2]^{1/2} < +\infty \quad \text{almost surely} ,
\]
where $\nu$ is a $[0,+\infty[$-valued random variable such that $\sup_{k \in \NN} \norm{x_k} \leq \nu < +\infty$. Then there exists a $[0,\infty[$-valued random variable $\kappa$ such that almost surely
\[
\norm{M^{x^-}_k} \leq \kappa \sqrt{\frac{\log(\log(N^{x^-}_k))}{N^{x^-}_k}} \quad \text{and} \quad 
\norm{M^{x}_k} \leq \kappa \sqrt{\frac{\log(\log(N^x_k))}{N^x_k}} .
\]
The same bound also holds for $\norm{M^x_k}$ replacing $N^{x^-}$ by $N^x$. For $\norm{M^y_k}$, we need first to bound $\norm{y_k}$ using almost sure boundedness of $x_k$. From \eqref{eq:algstepy}, we have
\begin{align*}
\norm{y_k} 
&\leq 2\norm{x_k}+ \norm{x_{k-1}} + \frac{1}{2L}\norm{\nabla f(x_k) + M^x_k} + \frac{1}{2L}\norm{\nabla f(x_{k-1}) + M^{x-1}_k} \\
&\leq \frac{5}{2}\norm{x_k}+ \frac{3}{2}\norm{x_{k-1}} + \frac{1}{L}\norm{\nabla f(0)} + \frac{1}{2L}\norm{M^x_k} + \frac{1}{2L}\norm{M^{x-}_k} .
\end{align*}
It then follows that $\norm{y_k}$ is almost surely bounded, and thus, there there exists a $[0,\infty[$-valued random variable $\kappa'$ such that
\[
\norm{M^y_k} \leq \kappa' \sqrt{\frac{\log(\log(N^y_k))}{N^y_k}} \quad \text{almost surely} .
\]
We now obtain that, almost surely,
\begin{multline*}
\norm{M_k} := \norm{\left(\beta_k\sqrt{s_{k}}t_{k+1}M^x_k - \beta_{k-1}\sqrt{s_{k-1}}t_k M^{x-}_k\right) + s_kt_{k+1}M^y_k} \\
\leq \kappa s_k t_{k+1}\sqrt{\frac{\log(\log(N^x_k))}{N^x_k}} + \kappa s_{k-1} t_k\sqrt{\frac{\log(\log(N^{x-}_k))}{N^{x-}_k}} + \kappa' s_k t_{k+1}\sqrt{\frac{\log(\log(N^y_k))}{N^y_k}} .
\end{multline*}

We will now proceed as in the proof of Theorem~\ref{thm:vals_thm} and all the following claims will hold almost surely (so we will avoid repeating the latter statement to lighten the exposition). We recall \eqref{eq:lyapunov_stoch} and embark from \eqref{eq:genrecursion2} to see that
\begin{equation*}
V_{k+1} - V_{k} 
\leq - \frac{\eta(1-\eta/4)}{4}s_kt_{k+1}^2 \left\|\nabla f(y_k)\right\|^2 + \norm{v_{k+1}}\epsilon_k + \epsilon_k^2 ,
\end{equation*}
where
\[
\epsilon_k = O\left(s_kt_{k+1} \sqrt{\frac{\log(\log(N^x_k))}{N^x_k}} + s_{k-1}t_k\sqrt{\frac{\log(\log(N^{x-}_k))}{N^{x-}_k}} + s_kt_{k+1} \sqrt{\frac{\log(\log(N^y_k))}{N^y_k}}\right) .
\]
Let $N_k=\min(N^x_k,N^{x-}_k,N^y_k)$. Assumption \eqref{eq:J0} thus reads
\[
\sum_{k \in \NN} s_k k\sqrt{\frac{\log(\log(N_k))}{N_k}} < +\infty ,
\]
which implies that $\sum_{k \in \NN} \epsilon_k^2 \leq (\sum_{k \in \NN} \epsilon_k)^2 < +\infty$. We then conclude that all statements from \ref{item:val_first_det} until \ref{item:vel_second} of Theorem~\ref{thm:vals_thm} hold almost surely. It remains to show that this is also the case for claim \ref{item:iter}.

From Theorem~\ref{thm:vals_thm}\ref{item:val_first_det}, we have \( f(x_k) \to \min_{\Hc} f \) almost surely. Let $\hat{\Omega}$ be the set of events on which the last statement holds and $\check{\Omega}$ on which boundedness of $(x_k)_{k \in \NN}$ holds (by assumption). Both sets are of probability one. For any $\omega \in \hat{\Omega} \cap \check{\Omega}$, let $(x_{k_j}(\omega))_{j \geq 1}$ be any converging subsequence, and $\bar{x}(\omega)$ its weak cluster point. We have
\[
f(\bar{x}(\omega)) = \lim_{j \to \infty} f(x_{k_j}(\omega)) = \lim_{k \to \infty} f(x_{k}(\omega)) = \min f ,
\]
which means that $\bar{x}(\omega) \in S$. This implies that almost surely, each weak cluster point of $(x_k)_{k \in \NN}$ belongs to $S = \argmin_{\Hc} f$. In other words, the first condition of Opial's lemma (Lemma~\ref{lemma:A4}) holds almost surely. We will now show that the second condition of this lemma also verified almost surely.

Recall \( h_k = \frac{1}{2} \|x_k - x^\star\|^2 \) for some $x^\star \in S$. Following the proof of Theorem~\ref{thm:vals_thm}\ref{item:vel_second}, we get that $h_k$ converges almost surely. However, it is important to realize that the set of events of probability $1$ on which this claim holds, exists depends on $x^\star$. To make this uniform on $S$ we use a separability argument, hence our assumption that $\Hc$ is also separable. 

We have just shown that there exists a set of events $\Omega_{x^\star}$ (that depends on $x^\star$) of probability one such that for all $\omega \in \Omega_{x^\star}$, $\lim_{k \to +\infty}\norm{x_k(\omega)-x^\star}$ exists. Since $\Hc$ is separable, there exists a countable set $U \subseteq S$, such that $\mathrm{cl}(U)=S$. Let $\tilde{\Omega}=\bigcap_{u \in U}\Omega_u$. Since $U$ is countable, a union bound shows that
\[
\mathbb{P}(\tilde{\Omega})=1-\mathbb{P}\left(\bigcup_{u \in U}\Omega_u^c\right)\geq 1-\sum_{u \in U}\mathbb{P}(\Omega_u^c) = 1.
\]
For arbitrary $x^{\star}\in S$, there exists a sequence $(u_j)_{j\in\NN} \subset U$ such that $u_j$ converges strongly to $x^{\star}$. Thus for every $j\in\NN$ there exists $\tau_j: \Omega_{u_j}\rightarrow\RR_+$ such that
\begin{equation}\label{eq.combinaison_eq_prec_12}
\lim_{k \to +\infty}\norm{x_k(\omega)-u_j}=\tau_j(\omega), \quad \forall\omega\in\Omega_{u_j}.
\end{equation}
Now, let $\omega \in \tilde{\Omega}$. Since $\tilde{\Omega} \subset \Omega_{u_j}$ for any $j \geq 1$, and  using the triangle inequality and \eqref{eq.combinaison_eq_prec_12}, we obtain that 
\[
\tau_j(\omega) - \norm{u_j-x^{\star}} \leq \liminf_{k \rightarrow +\infty}\norm{x_k(\omega)-x^{\star}} \leq \limsup_{k \to +\infty}\norm{x_k(\omega)-x^{\star}} \leq \tau_j(\omega) + \norm{u_j-x^{\star}} .
\]
Passing to $j \to +\infty$, we deduce
\[
\limsup_{j \to +\infty}\tau_j(\omega) \leq \liminf_{k \to +\infty}\norm{x_k(\omega)-x^{\star}} \leq \limsup_{k \to +\infty}\norm{x_k(\omega)-x^{\star}} \leq \liminf_{j \to +\infty}\tau_j(\omega) ,
\]
whence we deduce that $\lim_{j \to +\infty}\tau_j(\omega)$ exists for all $\omega \in \tilde{\Omega}$. The latter being independent of $x^\star$, we have just proved that almost surely, for any $x^\star \in S$, $\norm{x_k-x^{\star}}$ converges and its limit is equal to $\lim_{j \to +\infty}\tau_j$. 

We are finally in position to apply Opial's lemma at any $\omega \in \hat{\Omega} \cap \check{\Omega} \cap \tilde{\Omega}$, since $\mathbb{P}(\hat{\Omega} \cap \check{\Omega} \cap \tilde{\Omega})=1$, to conclude.
\end{proof}

\section{Numerical Illustrations}\label{s:num}
In this section, we compare three algorithms: our algorithm with $\beta > 0$ (S-IGAHD), $\beta=0$ (the stochastic version of Nesterov's method, S-FISTA), and the stochastic version of HBF (S-HBF). These algorithms were applied to machine learning problems involving the minimization of convex population risk for classification and regression. Randomness is brought by sampling batches (of increasing size) and evaluating the gradient on these batches. 

\subsection{Problem Specifications}

\subsubsection{Binary classification}
Each data sample consists of a six-dimensional input feature vector $\phib$ drawn from a multivariate normal distribution:
\begin{equation}\label{eq:X_class}
\phib \sim \mathcal{N}(m, \Sigma),
\end{equation}
where $m$ and $\Sigma$ denote the mean vector and covariance matrix, respectively. Each sample is assigned a binary output label $\yb$ taking values in $\{0,1\}$ which tells whether $\phib$ belongs to the first or second class.  

The classification model predicts the probability of the positive class using a linear transformation followed by a sigmoid activation function:
\begin{align}
h(x,\phib) = \sigma(\dprod{x,\phib}),
\end{align}
where $x \in \RR^6$ is a learnable weight vector, and $\sigma(\cdot)$ denotes the sigmoid function,
\begin{align}
\sigma(z) = \frac{1}{1 + e^{-z}}.
\end{align}

Conditioned on $x$ and $\phib$, $\yb$ is a Bernoulli random variables with success probability (or mean) $h(\dprod{\bar{x},\phib})$. By a standard maximum likelihood argument, the population loss, which represents the expected classification loss over the true data distribution, is given by
\begin{equation}\label{eq:logistic_pop}
\mathcal{R}_{\text{pop}}(x) = \mathbb{E}_{{(\phib, \yb) \sim \mu}} \left[ - \yb \log h(x,\phib) - (1 - \yb) \log (1 - h(x,\phib)) \right],
\end{equation}
where $\mu$ denotes the joint distribution of $\zeta=(\phib,\yb)$. The expectation is computed using numerical quadrature, specifically
tensor-product Gauss--Hermite quadrature:
\[
\mathcal{R}_{\mathrm{pop}}(\theta)
\approx \frac{1}{2\pi^{3/2}}
\sum_{i,j,k=1}^{n} w_i w_j w_k
\left[
\log\!\left(1+e^{x_{0,ijk}^{\top}\theta}\right)
+\log\!\left(1+e^{-x_{1,ijk}^{\top}\theta}\right)
\right],
\]
where \(x_{c,ijk}=\mu_c+\sqrt{2}\,L_c(t_i,t_j,t_k)^\top\),
\(L_cL_c^\top=\Sigma_c\), and \(t_i,w_i\) are the
Gauss--Hermite nodes and weights for the weight \(e^{-t^2}\).
Both classes have prior probability \(1/2\).

\subsubsection{Regression}
Each data sample consists of a six-dimensional input feature vector \( \phib \) drawn from a multivariate normal distribution,
    \begin{equation}\label{eq:X_reg}
        \phib \sim \mathcal{N}(m, \Sigma),
    \end{equation}
    The corresponding output observation is generated through a linear mapping,
    \begin{align*}
        \yb = M \phib,
    \end{align*}
    where \( M \) is a fixed regression coefficient matrix.
    
    To estimate the underlying relationship, we employ a linear regression model, predicting outputs as
    \begin{align*}
        \widehat{\yb} = X \phib,
    \end{align*}
    where \( X \) is a learnable regression coefficient matrix.
    
    The regression loss is formulated as the expected squared error under the data distribution,
    \begin{equation}\label{eq:reg_pop_loss}
        \mathcal{R}_{\text{pop}}(X) = \mathbb{E}_{(\phib, \yb) \sim \mu} \left[ \| X \phib - \yb \|^2 \right],
    \end{equation}
    where \( \mu \) represents the true underlying data distribution of $\zeta=(\phib,\yb)$. For regression, the population squared loss is approximated using
tensor-product Gauss--Hermite quadrature:
\[
\mathcal{R}_{\mathrm{pop}}(\theta)
=\mathbb{E}\!\left[(X^\top\theta-X^\top w^*)^2\right]
\approx \frac{1}{\pi^{3/2}}
\sum_{i,j,k=1}^{n} w_i w_j w_k
\left[x_{ijk}^\top(\theta-w^*)\right]^2,
\]
where \(x_{ijk}=\mu+\sqrt{2}\,L(t_i,t_j,t_k)^\top\),
\(LL^\top=\Sigma\), and \(t_i,w_i\) are the
Gauss--Hermite nodes and weights for the weight \(e^{-t^2}\).
For this quadratic integrand, the rule is exact up to numerical
precision when \(n\geq 2\).
    


\subsection{Training Specifications}
The minibatch sizes start at 2 and increase with iterations \(k\) as $N_{k} = 2 k^2$. For S-FISTA and {S-IGAHD}, we set \(\alpha = 3.1\), and for S-HBF, which has a constant damping factor, we use $\alpha = 0.1$. For all algorithms, we used a decreasing step-size strategy with $s_k = \frac{s_0}{k^{0.6}}$. The gradual increase of $N_k$ along with the decreasing step-size ensures the summability assumption in \eqref{eq:k2}; see the discussion in Remark~\ref{rem:conv_sto_2}. For S-IGAHD, the parameter \(\beta_k\) is taken very close, but strictly smaller than $\sqrt{s_k}/2$. Training is carried out over 200 epochs for both classification and regression problems.

\subsection{Results}
\begin{figure}[!ht]
    \centering
    \begin{subfigure}{0.45\textwidth}
        \includegraphics[width=\textwidth]{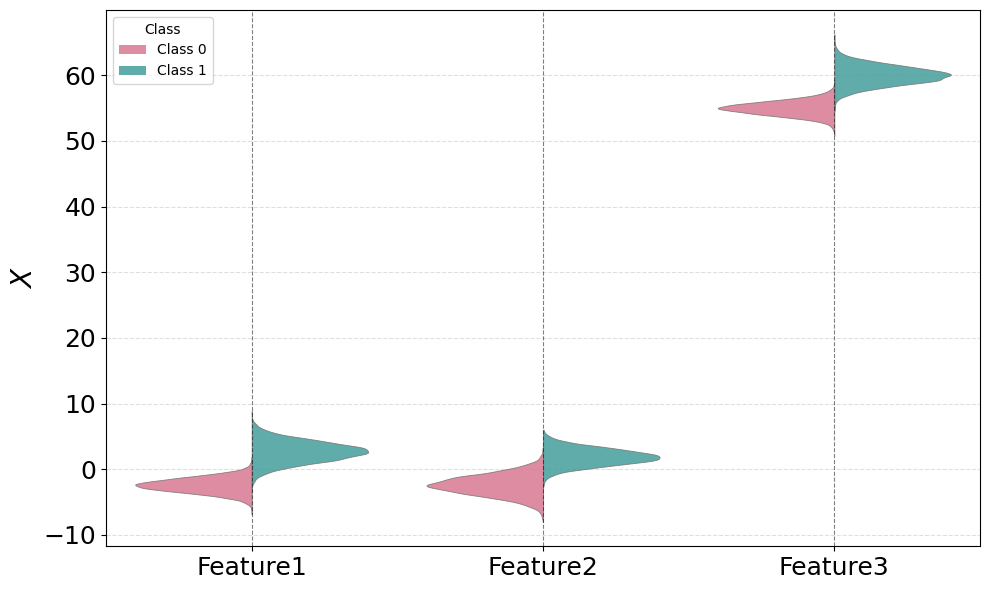}
        \caption{}
        \label{fig:ip_features1_subfig1}
    \end{subfigure}%
    \hfill
    \begin{subfigure}{0.45\textwidth}
        \includegraphics[width=\textwidth]{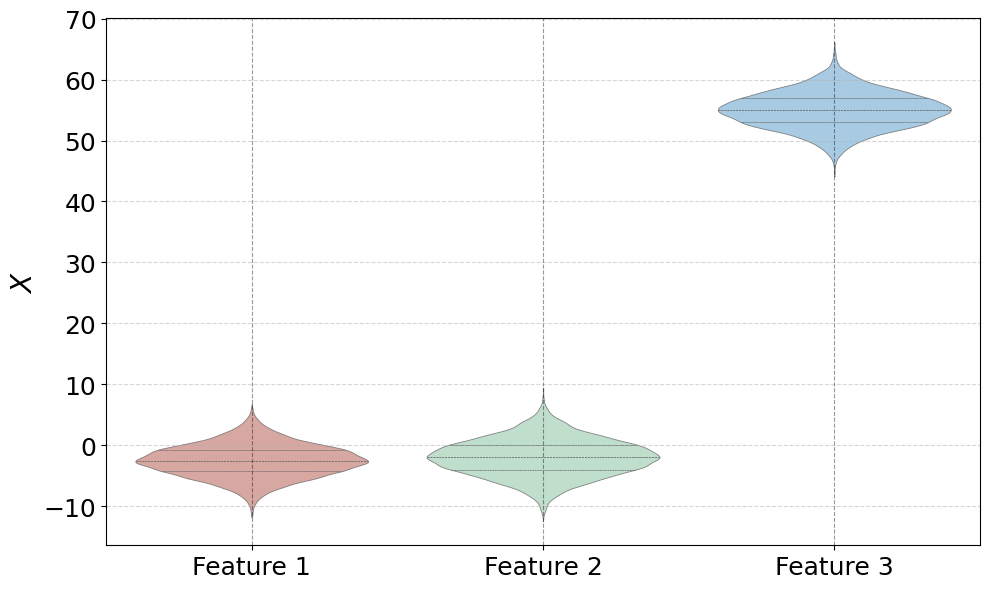}
        \caption{}
        \label{fig:ip_features2_subfig2}
    \end{subfigure}
    \caption{Plots showing representative distribution of individual features for (a) classification (b) regression problem where the full dataset is sampled from a multivariate Gaussian where feature 3 is disproportionally scaled w.r.t. rest of the features.}
    \label{fig:ip_data} 
\end{figure}

\begin{figure}[!ht]
    \centering
    \begin{subfigure}{0.45\textwidth}
        \includegraphics[width=\textwidth]{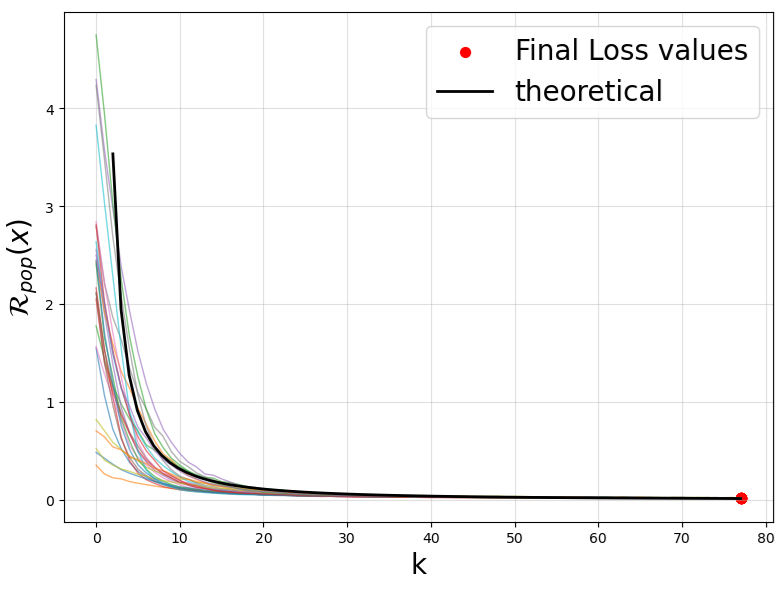}
        \caption{}
        \label{fig:IGAHD_conv_subfig1}
    \end{subfigure}%
    \hfill
    \begin{subfigure}{0.45\textwidth}
        \includegraphics[width=\textwidth]{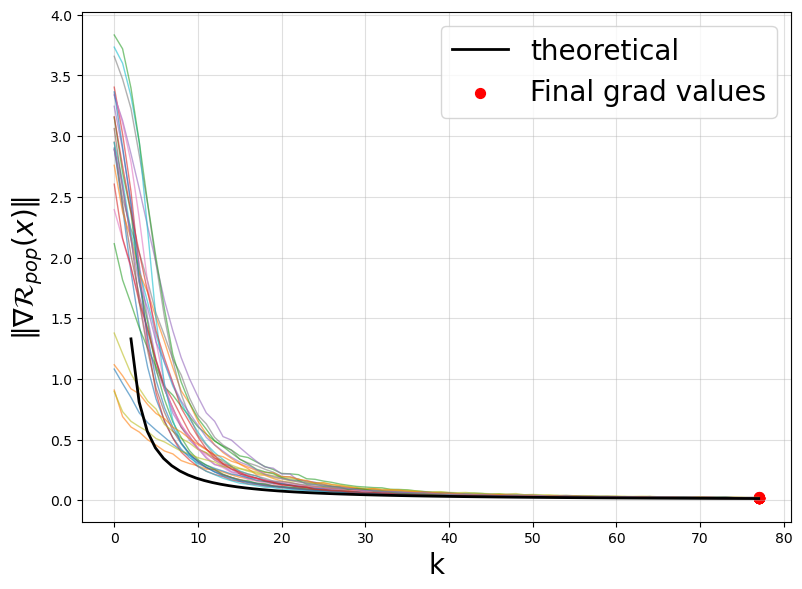}
        \caption{}
        \label{fig:IGAHD_conv_subfig2}
    \end{subfigure}
    \caption{The plots show  (a) the evolution of the sequence objective function which is the population risk over the entire distribution. (b) the gradient values for 25 randomly sampled initial weight vectors between (-1, 1) of dimension (6, 1) where 6 is the number of input features.}
    \label{fig:IGAHD} 
\end{figure}

In Figure~\ref{fig:IGAHD}, we observe that for all 25 independently and randomly initialized weight vectors, the population loss values converge to the minimal value. This behavior is consistent across all runs, with both the population loss and its gradient decreasing monotonically during training. 

Now we see the convergence rates for the three algorithms for a problem with an ill-conditioned objective.  Here we scale one of the features such that $\lambda_{\max}/\lambda_{\min} \sim 1000$, where $\lambda_{\max}$ and $\lambda_{\min}$ are the largest and smallest eigenvalues of the Hessian of feature matrix respectively. We show the trend of population loss and gradient of population loss along with the proven theoretical bounds. We can see that S-IGAHD with $\beta > 0$ converges faster and with less violent oscillations in population loss and gradients of population loss which is consistent with the observations in \cite{ACFR} (see also Section~\ref{subsec:role_beta} for more details). We see an even more pronounced trend in Figure~\ref{fig:comparison_reg} where S-IGAHD converges in objective values and gradients according to the theoretical rates after a few initial oscillations, whereas the other two algorithms fail to converge at all within the allotted training period.
\begin{figure}[!ht]
    \centering
    \begin{subfigure}{0.45\textwidth}
        \includegraphics[width=\textwidth]{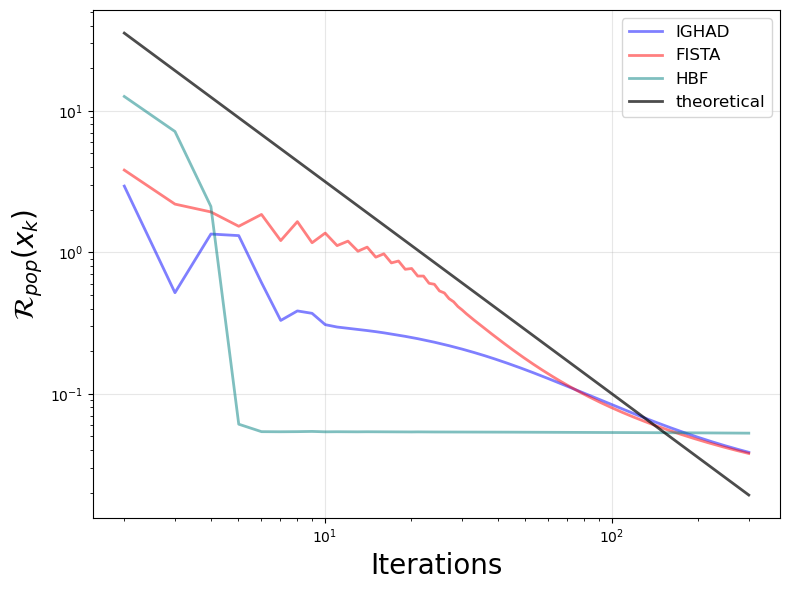}
        \caption{}
        \label{fig:comp_conv_subfig1}
    \end{subfigure}%
    \hfill
    \begin{subfigure}{0.45\textwidth}
        \includegraphics[width=\textwidth]{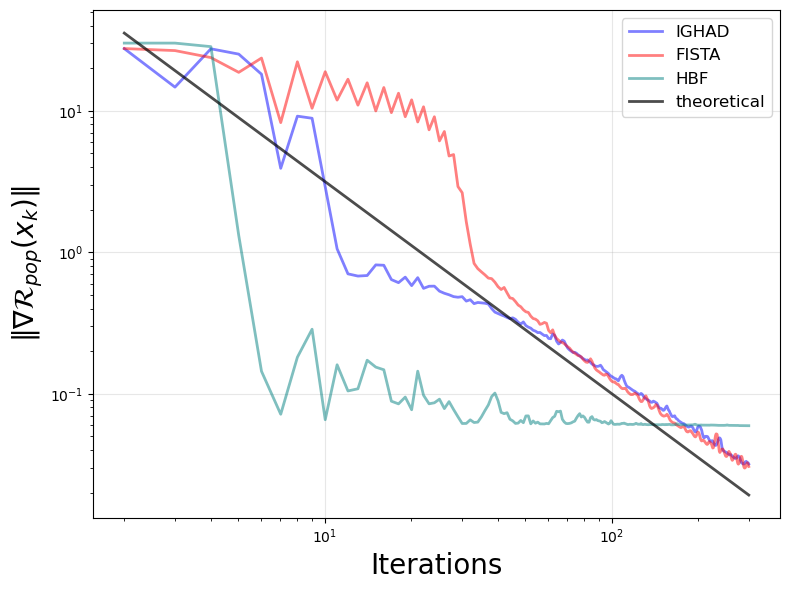}
        \caption{}
        \label{fig:comp_conv_subfig2}
    \end{subfigure}
    \caption{For the classification problem: (a) convergence of objective values (population loss) for the three algorithms, and (b) the gradient values such that $\lambda_{\max}/\lambda_{\min} \sim 1000$, along with the theoretical bounds proven above for S-IGAHD.}
    \label{fig:comparison_class} 
\end{figure}

\begin{figure}[!ht]
    \centering
    \begin{subfigure}{0.45\textwidth}
        \includegraphics[width=\textwidth]{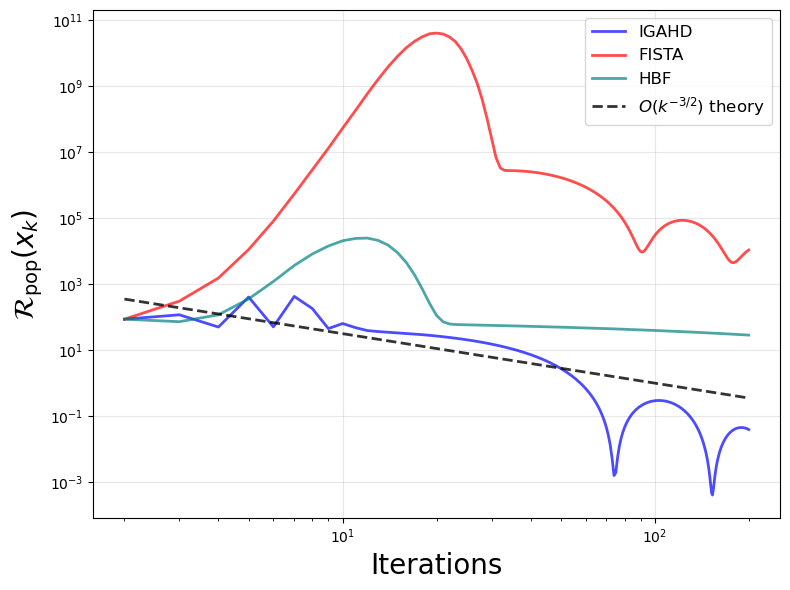}
        \caption{}
        \label{fig:comp_conv_subfig_reg_1}
    \end{subfigure}%
    \hfill
    \begin{subfigure}{0.45\textwidth}
        \includegraphics[width=\textwidth]{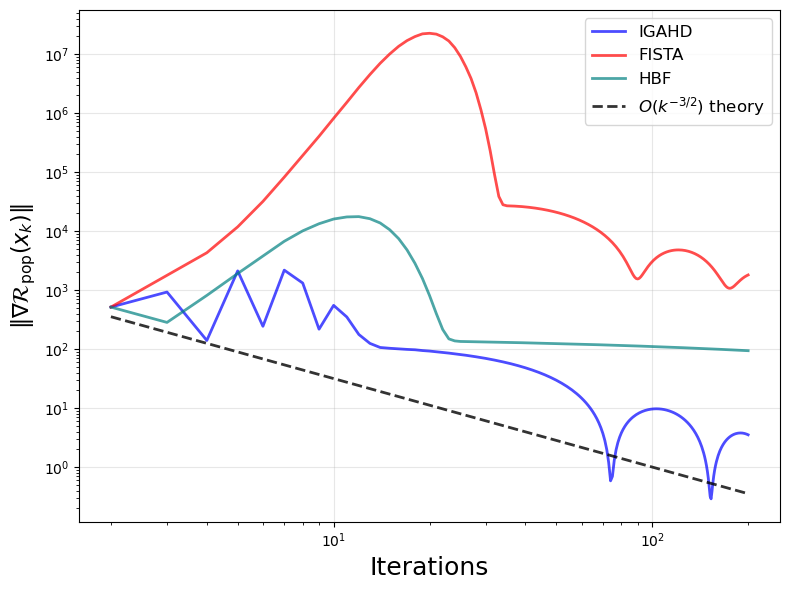}
        \caption{}
        \label{fig:comp_conv_subfig_reg_2}
    \end{subfigure}
    \caption{For the regression problem: (a) convergence of objective values (population loss) for the three algorithms, and (b) gradient values 
    such that $\lambda_{\max}/\lambda_{\min} \sim 1000$, along with the theoretical bounds proven above for S-IGAHD.}
    \label{fig:comparison_reg} 
\end{figure}

\section{Conclusion}\label{s:conc}
It has been observed in the smooth deterministic and exact setting that incorporating a damping term driven by the Hessian into inertial first-order optimization algorithms can significantly reduce oscillations and fast convergence of the objective and gradients, which works particularly well for ill-conditioned objectives. This is in stark contrast with other accelerated first-order methods, which often exhibit persistent oscillatory behavior under such conditions. In this work, we extend these insights to the inexact case, both in the deterministic and stochastic settings. Our algorithm is first-order and does not involve computing the Hessian explicitly whatsoever. Our results demonstrate the algorithm's practical utility in machine learning applications, particularly in settings where ill-conditioning is prevalent, such as in convex regression problems.
We prove theoretical fast convergence rates on the objective, the gradient and velocities, both in the deterministic and stochastic settings, provided that the errors vanish fast enough. This holds for instance in the stochastic case by balancing the choice of the step-size sequence and that of the minibatch size sequence. While we also prove convergence of the iterates themselves in the deterministic case, extending such result to the stochastic case is much more challenging due to intricate dependencies between the gradient estimates. Our work lays the foundation for further exploration of Hessian-driven first-order methods in noisy or data-driven optimization, particularly in large-scale machine learning problems where ill-conditioning is intrinsic. Beyond proving convergence of the iterates, future work may include designing adaptive variants of the algorithm and variance-reduced versions for finite sums.

\section*{Acknowledgements}{HC and VK acknowledge support from the Czech National Science Foundation under Project 24-11664S}

\bibliographystyle{plain}
\bibliography{refs}

\appendix
\renewcommand{\thesection}{A}  
\section*{Auxiliary Lemmas A}
\begin{lemma}[Extended Descent Lemma]\label{lemma:A1} 
Let $g: \Hc \rightarrow \RR$ be a convex function whose gradient is $L$-Lipschitz continuous. Let $s \in [0, 1/L]$. Then for all $(x, y) \in \Hc^{2}$, we have
\[
g(y - s \nabla g(y)) \leq g(x) + \langle \nabla g(y), y - x \rangle - \frac{s}{2} \|\nabla g(y)\|^{2} - \frac{s}{2} \|\nabla g(x) - \nabla g(y)\|^{2}.
\]
\end{lemma}
See \cite[Lemma~1]{ACFR} for a proof.

\begin{lemma}[Discrete Gronwall-Bellman Inequality \cite{brezis2011functional}]\label{lemma:A2}
Let $(a_{k})_{k \in \NN}$ be a sequence of nonnegative numbers such that
\[
a_{k}^{2} \leq c^{2} + \sum_{j=1}^{k} \beta_{j} a_{j}
\]
for all $k \in \NN$, where $(\beta_{k})_{k \in \NN}$ is a summable sequence of nonnegative numbers, and $c \geq 0$. Then
\[
a_{k} \leq c + \sum_{j=1}^{\infty} \beta_{j} \quad \text{for all } k \in \NN.
\]
\end{lemma}

he following lemma is standard and can be found in e.g., \cite{Combettes01}.
\begin{lemma}\label{lemma:A3-det}
Let $(a_{k})_{k \in \NN}$ be a real sequence which is bounded from below such that
\[
a_{k+1} \leq a_k - b_k + \beta_k
\]
for all $k \in \NN$, where $(\beta_{k})_{k \in \NN}$ is a nonnegative real sequence and $(\beta_{k})_{k \in \NN}$ is nonnegative summable real sequence. Then $\sum_{k \in \NN} b_k < +\infty$ and $a_k$ converges.
\end{lemma}

A simple contradiction argument yields the following result.
\begin{lemma}\label{lemma:Asumo}
Let $(c_k)_{k\in \NN}$ and $(d_k)_{k\in \NN}$ be two nonnegative sequences such that $\sum_{k\in \NN}c_kd_k < +\infty$ and $\sum_{k\in \NN}c_k = +\infty$. Then $\liminf_{k \to +\infty} d_k = 0$. 
\end{lemma}

\begin{lemma}\label{lemma:A6}
Given $\alpha > 1$, let $(a_k)_{k\in \NN}$, $(\omega_k)_{k \in \NN}$ be two sequences of nonnegative numbers such that
\[
a_{k+1} \le \left(1 - \frac{\alpha}{k}\right)a_k + \omega_k
\]
for all $k \ge 1$. If $\sum_{k\in \NN}k\omega_k < +\infty, then \sum_{k \in \NN}a_k < +\infty$ and $a_k=o(1/k)$.
\end{lemma}

\begin{proof}
Multiplying both sides by $k+1$, we obtain
\[
(k+1)a_{k+1} \leq ka_k - (\alpha-1)a_k + (k+1)\omega_k .
\]
Lemma~\ref{lemma:A3-det} yields the summability claim and $\lim_{k \to +\infty} ka_k$ exists. Applying Lemma~\ref{lemma:Asumo} with $c_k=1/k$ and $d_k=ka_k$ gives $\lim_{k \to +\infty} ka_k=0$.
\end{proof}

\begin{lemma}[Convergence of Non-negative Almost Supermartingales \cite{robbins1971convergence}]\label{lemma:A3}
Given a filtration $\mathscr{R} = \left(\mathcal{R}_{k}\right)_{k \in \NN}$ and the sequences of real-valued random variables $\left(r_{k}\right)_{k \in \NN} \in \ell_{+}(\mathscr{R})$, $\left(a_{k}\right)_{k \in \NN} \in \ell_{+}(\mathscr{R})$, and $\left(z_{k}\right)_{k \in \NN} \in \ell_{+}^1(\mathscr{R})$ satisfying, for each $k \in \NN$,
\[
\EE\left[r_{k+1} \mid \mathcal{R}_{k}\right] - r_{k} \leq -a_{k} + z_{k} \quad (\mathbb{P}\text{-a.s.}),
\]
it holds that $\left(a_{k}\right)_{k \in \NN} \in \ell_{+}^1(\mathscr{R})$ and $\left(r_{k}\right)_{k \in \NN}$ converges $(\mathbb{P}\text{-a.s.})$ to a random variable valued in $[0, +\infty)$.
\end{lemma}

\begin{lemma}\label{lemma:meanas}
Let $(X_k)_{k \in \NN}$ be a sequence of random variables such that $\sum_{k \in \NN} \EE[|X_k|] < +\infty$. Then $X_k \to 0 ~ \Pas$.
\end{lemma}
\begin{proof}
The proof is standard using Markov's inequality and Borel-Cantelli lemma. Indeed, for any $\varepsilon >0$, we have 
\[
\PP(|X_k| > \varepsilon) \leq \frac{\EE[|X_k|]}{\varepsilon} .
\]
Thus
\[
\sum_{k \in \NN}\PP(|X_k| > \varepsilon) \leq \frac{\sum_{k \in \NN}\EE[|X_k|]}{\varepsilon} < +\infty ,
\]
and we conclude by the Borell-Cantelli lemma.
\end{proof}

\begin{lemma}[Opial's lemma \cite{opial1967theorem}]\label{lemma:A4}
Let $S$ be a nonempty subset of $\Hc$ and $\left(x_{k}\right)$ be a sequence of elements of $\Hc$. Assume that:
\begin{enumerate}[label=(\roman*), ref=\roman*]
    \item\label{item:opial1} Every sequential weak cluster point of $\left(x_{k}\right)$, as $k \rightarrow \infty$, belongs to $S$.
    \item\label{item:opial2} For every $z \in S$, $\lim_{k \rightarrow \infty} \left\|x_{k} - z\right\|$ exists.
\end{enumerate}
Then $\left(x_{k}\right)$ weakly converges as $k \rightarrow +\infty$ to a point in $S$.
\end{lemma}

\appendix
\renewcommand{\thesection}{B}
\section*{Auxiliary Results B}
\begin{subsection}{Closed form solutions of \eqref{eq:dinavd} for quadratic functions} \label{subsec:closed_form_DIN_AVD}
We consider a slightly general version of \eqref{eq:isehd} that we write
\begin{equation}\label{eq:dinavd}\tag{$\mathrm{DIN-AVD}$}
\ddot{x}(t) + \displaystyle{\frac{\alpha}{t}}\dot{x}(t) +  \beta(t) \frac{d}{dt}\nabla  f (x(t)) + b(t) \nabla  f (x(t)) = 0.
\end{equation}
In \cite{ACFR}, the authors derived closed-form solutions to \eqref{eq:dinavd}
for a quadratic objective function 
$f: \mathbb{R}^n \to \langle A x, x \rangle$, which is responsible for the observed exponential asymptotic convergence for $\beta>0$. Here $A$ is a symmetric positive definite matrix. The case of a positive semidefinite matrix $A$ can be treated similarly by restricting the analysis to $\ker(A)^\perp$.
Projecting \eqref{eq:dinavd} onto the eigenspace of $A$, one has to solve $n$ independent one-dimensional ODEs of the form:
\begin{equation}
    \ddot{x}_i(t) + \left( \frac{\alpha}{t} + \beta(t) \lambda_i \right) \dot{x}_i(t) + \lambda_i b(t) x_i(t) = 0, \quad i = 1, \dots, n,
\end{equation}
where $\lambda_i > 0$ is an eigenvalue of $A$. In the following, we drop the subscript $i$.

\paragraph{Case:} $\beta(t) \equiv \beta$, $b(t) = b + \frac{\gamma}{t}$, with $\beta \geq 0$, $b > 0$, $\gamma \geq 0$. The ODE then reads:

\begin{equation}\label{eq:IGAHD_eq}
\begin{aligned}
    \ddot{x}(t) + \left( \frac{\alpha}{t} + \beta \lambda \right) \dot{x}(t) + \lambda \left( b + \frac{\gamma}{t} \right) x(t) = 0.
    \end{aligned}
\end{equation}
\begin{itemize}

\item If $\beta^2 \lambda^2 \neq 4b \lambda$, we set:
\begin{align*}
    \xi &= \sqrt{\beta^2 \lambda^2 - 4b \lambda},\quad 
    \kappa = \lambda \frac{\gamma - \alpha \beta / 2}{\xi},\quad 
    \sigma = \frac{\alpha - 1}{2}.
\end{align*}
The solution to the ODE can be written as:
\begin{equation}\label{eq:hyper_geom}\begin{aligned}
    x(t) = \xi^{\alpha/2} e^{- (\beta \lambda + \xi) t / 2} \left[ c_1 M\left( \frac{\alpha}{2} - \kappa, \alpha, \xi t \right) + c_2 U\left( \frac{\alpha}{2} - \kappa, \alpha, \xi t \right) \right],
    \end{aligned}
\end{equation}
where $c_1$ and $c_2$ are constants determined by the initial conditions. $M$ and $U$ are respectively the Whittaker functions and Kummer's confluent hypergeometric functions.

\item If $\mathcal{\beta}^2 \lambda^2 = 4b \lambda$, set $\zeta = 2\sqrt{\lambda \left(\gamma - \frac{\alpha \beta}{2} \right)}$. The solution to \eqref{eq:IGAHD_eq} takes the form:
\begin{align*}
    x(t) = t^{-\frac{(\alpha -1)}{2}} e^{-\frac{\beta \lambda t}{2}} \left[c_1 J_{(\alpha -1)/2}(\zeta \sqrt{t}) + c_2 Y_{(\alpha -1)/2}(\zeta \sqrt{t}) \right],
\end{align*}
where $J_{
\nu}$ and $Y_{
\nu}$ are the Bessel functions of the first and second kind. When $\beta > 0$, one can clearly see the exponential decrease brought by the Hessian damping. 

Moreover, for large $t$, $|x(t)|$ behaves as:
    \item If $\beta^2 \lambda^2 > 4b \lambda$, we have:
    \begin{align*}
        |x(t)| = \mathcal{O} \left( t^{-\frac{\alpha}{2} + |\kappa|} e^{-\frac{\beta \lambda - \xi}{2} t} \right)
        = \mathcal{O} \left( e^{-\frac{2b}{\beta} t -\left(\frac{\alpha}{2} - |\kappa| \right) \log t }\right).
    \end{align*}
    \item If $\beta^2 \lambda^2 < 4b \lambda$, whence $\xi \in i \mathbb{R}_*^+$ and $\kappa \in i \mathbb{R}$, we have:
    \begin{align*}
        |x(t)| = \mathcal{O} \left( t^{-\frac{\alpha}{2}} e^{-\frac{\beta \lambda}{2} t} \right).
    \end{align*}
    \item If $\beta^2 \lambda^2 = 4b \lambda$, we have:
    \begin{align*}
        |x(t)| = \mathcal{O} \left( t^{-\frac{2\alpha -1}{4}} e^{-\frac{\beta \lambda}{2} t} \right).
    \end{align*}
\end{itemize}
\end{subsection}

\begin{subsection}{Role of positive damping parameter \texorpdfstring{$\beta_k$}{Lg} in convergence}
\label{subsec:role_beta}
We now show through numerical results that the Hessian damping parameter $\beta_k$ is indeed responsible for exponential asymptotic convergence of values. We consider the case of regression loss in \eqref{eq:reg_loss}  where in we investigate two cases where in we initialize $\beta$ to a small positive value $(10^{-8})$ where in case (a) it stays constant (b) increases quadratically w.r.t. iteration number $k$, s.t. $\beta_k\le\frac{\sqrt{s_k}}{2} $ for all $k$. 
 Figure~\ref{fig:comp_xi2} shows the two cases. One can clearly see the dampening of oscillations and rapid convergence as $\xi^2$ which depends on $\beta_k$ becomes positive in (b) which drives the exponential decay, where in contrast in case (a) the oscillations persist as $\xi^2$ is always negative which is in agreement with \eqref{eq:hyper_geom}. We can see that the loss still decreases in both cases but in (a) with higher oscillations which is due to the fact that \eqref{eq:hyper_geom} has an oscillation factor introduced by a negative $\xi^2$ as seen in case (a) whereas in case (b) we get an exponential decay factor as $\xi^2$ becomes positive.
\begin{figure}[htbp]
    \centering
    \begin{subfigure}{0.45\textwidth}
        \includegraphics[width=\textwidth]{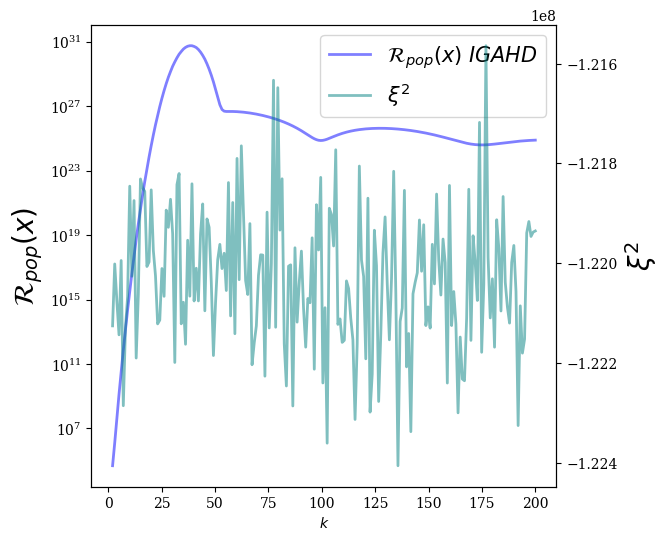}
        \caption{}
        \label{fig:xi_2_neg_subfig1}
    \end{subfigure}%
    \hfill
    \begin{subfigure}{0.45\textwidth}
        \includegraphics[width=\textwidth]{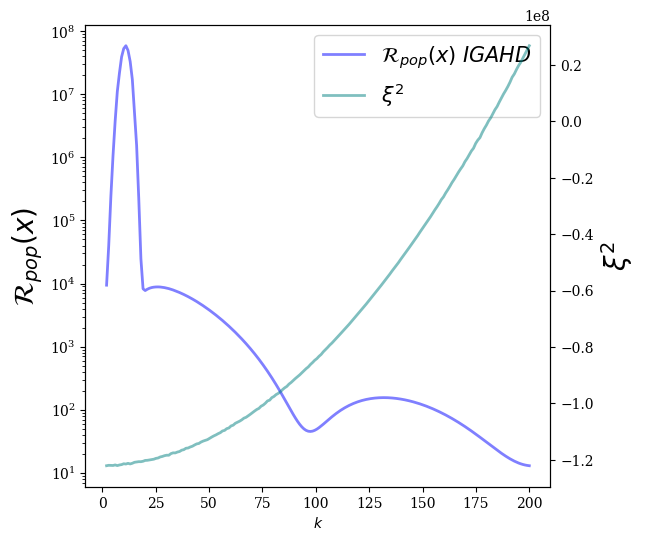}
        \caption{}
        \label{fig:xi_2_pos_subfig1}
    \end{subfigure}
    \caption{Comparison of convergence of objective value  with (a) $\beta_k\sim 0$ (similar to FISTA) and (b) non-zero decreasing $\beta_k$ s.t. $0 < \beta_k < \frac{\sqrt{s_k}}{2}$.}
    \label{fig:comp_xi2} 
\end{figure}
\end{subsection}

\begin{subsection}{Role of step-size \texorpdfstring{$s_k$}{Lg} in  convergence}
\label{subsec:role_step-size}
In Theorem~\ref{thm:vals_thm}\ref{item:vel}, we assumed a constant step-size bounded away from $0$ to derive rates on velocities, this can broadly be understood by considering the way the Hessian-driven damping term influences the dynamics in the continuous system
\begin{equation}\label{eq:cont_IGAHD}
    \begin{aligned}
        \ddot{x}(t) + \alpha(t)\dot{x}(t) + \beta(t)\nabla^2f(x(t))\dot{x}(t) + b(t)\nabla f (x(t)) = 0 .
    \end{aligned}
\end{equation}
The Hessian-driven damping term introduces a force given by $F^d = \beta(t)\nabla^2f(x(t))\dot{x}(t)$, when we use the discretized version of the ODE as given by Algorithm \ref{alg:hdde} with a step-size $s_k$, the damping force at any iteration $k$ can be described as
\begin{equation}\label{discreet_damp}
    \begin{aligned}
        F^d_k = \beta_k \sqrt{s_k} \nabla f(x_k) - \beta_{k-1} \sqrt{s_{k-1}} \nabla f(x_{k-1}) .
    \end{aligned}
\end{equation} for this term to physically represent a damping force, the above equation should satisfy monotonicity conditions at all times, which can't be guaranteed unless the step-size $s_k$ is constant. In the stochastic case, we prove the convergence of values and gradients but not iterates and step-size plays an important role in achieving the rates. Specifically, it ensures that the summability conditions in \eqref{eq:k2} are satisfied. However, the step-size is not the only way to control the variance, where it can also be controlled by using an increasing batch-size but is useful in case of problems with limited data availability. For the same regression problem above, we analyze two cases where: in case~1, we use a less aggressive step-size reduction where $s_k \sim \frac{s_0}{k^{0.2}}$, whereas in case~2, the step-size decreases more rapidly as $s_k \sim \frac{s_0}{k^{0.6}}$, In Figure~\ref{fig:comp_sk}(a) we see that the quantity $(s_k k \sigma_{x,k})$ decreases slower than $\mathcal{O}(1/k)$ rate and hence we see that the objective values do not converge at the theoretically proven rates whereas (b) clearly shows that the derived theoretical rates are achieved when the quantity is summable follows the summability relation \eqref{eq:k2}.

\begin{figure}[!ht]
    \centering
    \begin{subfigure}{0.45\textwidth}
        \includegraphics[width=\textwidth]{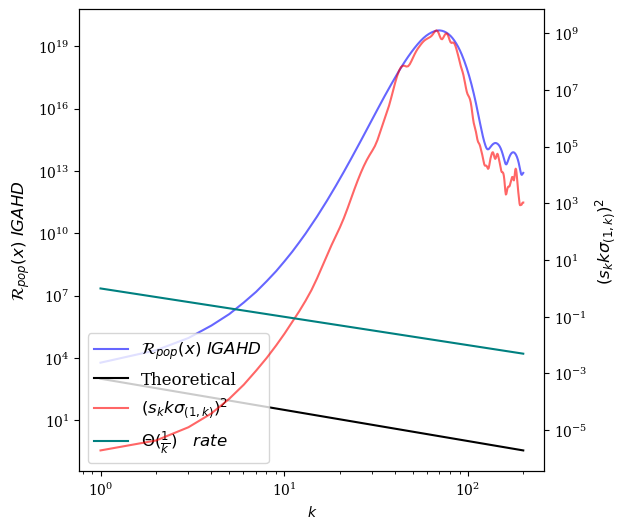}
        \caption{}
        \label{fig:unbounded_l2_subfig_reg_1}
    \end{subfigure}%
    \hfill
    \begin{subfigure}{0.45\textwidth}
        \includegraphics[width=\textwidth]{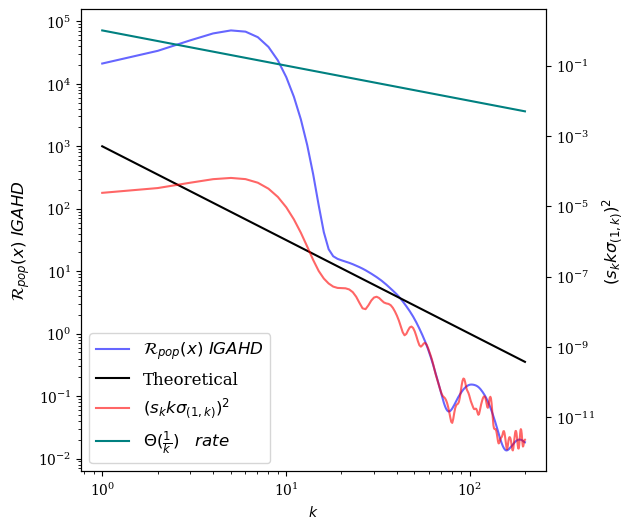}
        \caption{}
        \label{fig:unbounded_l2_subfig_reg_2}
    \end{subfigure}
    \caption{Representatve plots showing the evolution of population loss for 2 different step-size schemes   (a) $s_k = \frac{s_0}{k^{0.2}}$ (b)$s_k = \frac{s_0}{k^{0.6}}$ along with the variance term, when the stochastic variance sequence decreases faster than the $\Theta(\frac{1}{k})$ rates, the observed rate for the objective matches the theoretical one. }
    \label{fig:comp_sk} 
\end{figure}
\end{subsection}

\end{document}